\documentclass[reqno,11pt]{amsart}
\usepackage{ytableau}
\usepackage{amsthm,amsfonts, amssymb, amscd}
\usepackage[all,cmtip]{xy}
\usepackage{euscript}

\usepackage{tikz}
\usetikzlibrary{matrix,arrows}
\usetikzlibrary{positioning}

\usepackage[normalem]{ulem}

\usepackage{hyperref}
\hypersetup{%
    linktoc=page
}

\let\oldtocsection=\tocsection
\let\oldtocsubsection=\tocsubsection
\renewcommand{\tocsection}[2]{\hspace{0em}\oldtocsection{#1}{#2}}
\renewcommand{\tocsubsection}[2]{\hspace{1em}\oldtocsubsection{#1}{#2}}

\newtheorem{Thm}{Theorem}[section]
\newtheorem{Lem}[Thm]{Lemma}
\newtheorem{Cor}[Thm]{Corollary}
\newtheorem{Prop}[Thm]{Proposition}

\theoremstyle{remark}
\newtheorem{Rem}[Thm]{Remark}
\theoremstyle{remark}
\newtheorem{example}[Thm]{Example}
\theoremstyle{definition}

\theoremstyle{definition}

\theoremstyle{definition}

\numberwithin{equation}{section}

\newcommand{\A}{\mathbb{ A}}

\newcommand{\R}{\mathbb{ R}}           
\newcommand{\C}{\mathbb{C}}           
\newcommand{\Z}{\mathbb{ Z}}           

\newcommand{\QQ}{\mathbb{ Q}}        
\newcommand{\ad}{\operatorname{ad}}             

\newcommand{\spec}{\operatorname{Spec}}

\newcommand{\Ind}{\operatorname{Ind}}

\renewcommand{\ker}{\operatorname{ker }}

\newcommand{\fb}{{\mathfrak b}}

\newcommand{\fg}{{\mathfrak g}}

\newcommand{\fl}{{\mathfrak l}}

\newcommand{\fp}{{\mathfrak p}}

\newcommand{\ft}{{\mathfrak t}}
\newcommand{\fu}{{\mathfrak u}}

\newcommand{\ga}{\alpha}
\newcommand{\gb}{\beta}

\newcommand{\gre}{\epsilon}

\newcommand{\gl}{\lambda}

 \newcommand{\cb}{\mathcal{B}}

\newcommand{\cf}{\mathcal{F}}

\newcommand{\ch}{\mathcal{H}}

  \newcommand{\ck}{\mathcal{K}}
 \newcommand{\cl}{\mathcal{L}}
 \newcommand{\cm}{\mathcal{M}}
 \newcommand{\cn}{\mathcal{N}}
 \newcommand{\co}{\mathcal{O}}

 \newcommand{\cv}{\mathcal{V}}

\newcommand{\toric}{\mathcal{V}}

\newcommand{\map}{\eta}

\newcommand{\shq}{\underline{\QQ}}
\newcommand{\omunder}{\underline{\omega}}

\newcommand{\nilp}{\nu}

\renewcommand{\tilde}{\widetilde}

\renewcommand{\bar}[1]{\overline{#1}}

\newcounter{myenumi}
\renewcommand{\themyenumi}{$(\arabic{myenumi})$}

\title[The generalized Springer correspondence]{A new approach to the generalized Springer correspondence}

\author{William Graham}
\address{
Department of Mathematics\\ University of Georgia\\ Boyd Graduate Studies Research Center\\ Athens, GA\\ 30602\\ USA
}
\email{wag@math.uga.edu}

\author{Martha Precup}
\address{Department of Mathematics and Statistics\\ Washington University in St. Louis \\ One Brookings Drive\\ St. Louis, Missouri\\ 63130\\ USA  }
\email{martha.precup@wustl.edu}

\author{Amber Russell}
\address{
Department of Mathematics, Statistics, and Actuarial Science,
Butler University, 
4600 Sunset Avenue, Indianapolis, Indiana 46208\\ USA}
\email{acrusse3@butler.edu}

\date{\today}

\begin{document}

\begin{abstract} The Springer resolution of the nilpotent cone is used to give a geometric construction of the irreducible representations of Weyl groups. Borho and MacPherson obtain the Springer correspondence by applying the decomposition theorem to the Springer resolution, establishing an injective map from the set of irreducible Weyl group representations to simple equivariant perverse sheaves on the nilpotent cone.  In this manuscript, we consider a generalization of the Springer resolution using a variety defined by the first author.  Our main result shows that in the type A case, applying the decomposition theorem to this map yields all simple perverse sheaves on the nilpotent cone with multiplicity as predicted by Lusztig's generalized Springer correspondence.  
\end{abstract}

\maketitle

\section{Introduction}

In 1976, Springer introduced a geometric construction of the irreducible representations of the Weyl group
of a semisimple algebraic group $G$ on the cohomology of algebraic varieties called Springer fibers
\cite{Springer1976, Springer1978}.  Springer's work was foundational to the field of geometric representation theory.  There are many current research directions connected to Springer's classical result, 
ranging from exploring modular versions of his work \cite{AHJR1, AHJR2}, to applying geometry, topology, and combinatorics to better understand Springer fibers \cite{FMS15, Wilbert, Kim19}.

The Springer fibers are the fibers of the Springer resolution
\begin{eqnarray*}
\mu: \tilde{\cn} \to \cn
\end{eqnarray*}
which is a resolution of singularities of the nilpotent cone $\cn$ of  $\fg = \mbox{Lie }G$.  In this manuscript, we consider an extension of the Springer resolution, which is a map
\begin{eqnarray*}
\psi = \tilde{\map}\circ \mu: \tilde{\cm} \xrightarrow{\;\tilde{\map}\;} \tilde{\cn} \xrightarrow{\;\mu\;} \cn,
\end{eqnarray*}
from a variety $\tilde{\cm}$ to the nilpotent cone that factors through the Springer resolution.  We refer to this map as the 
extended Springer resolution.  The variety $\tilde{\cm}$ was defined by the first author in \cite{Gra:19}, and used
to construct an analogue of the Springer resolution for $\cm = \spec R(\tilde{\co}^{pr})$.  Here $R(\tilde{\co}^{pr})$ denotes the ring of regular functions on $\tilde{\co}^{pr}$, the simply connected cover of the principal nilpotent orbit $\co^{pr} \subset \cn$.
Although $\tilde{\cm}$ is not smooth, it is locally a quotient of a smooth variety by a finite group. The variety $\tilde{\cm}$ and the map $\psi$ are defined precisely in Section~\ref{section: generalized resolution} below.

When viewed through the lens of the derived category of $G$-equivariant perverse sheaves, the Springer correspondence is an {injective map} from the set of irreducible representations of the Weyl group to the set of simple $G$-equivariant perverse sheaves on the nilpotent cone.   The latter set is indexed by the pairs $(\co, \cl)$, where $\co$ is a nilpotent orbit and $\cl$ is an irreducible $G$-equivariant local system on $\co$.  Since the Springer correspondence is an injection, we recover {only some} of the pairs $(\co, \cl)$ in this manner. Those missing from the Springer correspondence were described by Lusztig. He established a bijection between irreducible representations of what he called relative Weyl groups and all simple $G$-equivariant perverse sheaves on $\cn$ \cite{Lu84}, proving that the Springer correspondence is a special case of what is now known as the generalized Springer correspondence.  Historically, this result was the beginning of Lusztig's consideration of character sheaves, a key tool in the study of irreducible representations of finite groups of Lie type. A summary of this work can be found in~\cite{carter2006}.

Unlike the proof of the Springer correspondence given by Borho and MacPherson \cite{Borho1981}, Lusztig's 
proof of the generalized Springer correspondence does not proceed by applying the decomposition theorem to the pushforward of the constant sheaf on a single variety.  
Our main result, stated precisely in Theorem~\ref{thm.genspringerA} below, 
applies this strategy in the case of $G=SL_n(\C)$ to the extended Springer resolution $\tilde{\cm} \to \cn$.  
This theorem shows that the pushforward of the constant sheaf on
$\tilde{\cm}$ to the nilpotent cone $\cn$ is the direct sum of all the so-called {Lusztig sheaves} used in the generalized
Springer correspondence.  This then implies that the pushforward is the direct sum of all the $G$-equivariant perverse sheaves on $\cn$, each occurring with multiplicity as determined by the generalized Springer correspondence. The fact that all $G$-equivariant perverse sheaves appear here was obtained by the third author in \cite{Russell}, using different methods and without the multiplicity result.  
 
We believe that the geometry of the extended Springer resolution will yield further insights into both Springer fibers and the generalized Springer correspondence.  
In forthcoming work, the authors will characterize irreducible components of the fibers of 
the map $\psi:\tilde{\cm} \to \cn$, and use the geometric information to analyze the generalized Springer correspondence from this new perspective.  The generalized Springer correspondence
is related to relative Weyl groups, each of which is a smaller symmetric group, and acts on the cohomology of $\psi^{-1}(x)$ for certain $x\in \cn$.  Since the map $\psi$ factors through the Springer resolution, our construction yields a new connection between the representation theory of these smaller symmetric groups and the geometry of Springer fibers.  This connection sheds
light on the Springer fibers themselves and yields previously unknown geometric features
of these fibers.

The arguments below rely heavily on the structure theory of the derived category of constructible sheaves.
The proof of our main result is comprised of two main steps.  In broad terms, the first step requires us to analyze the decomposition theorem when applied to a quotient map $X\to X/Z$ where $Z$ is a finite abelian group acting on a variety $X$.  We then apply this general analysis in the special case of $\tilde{\map}: \tilde{\cm} \to \tilde{\cn}$.  The second step of our argument is comprised of studying the simple perverse sheaves that appear as summands of $\tilde{\map}_*\shq_{\tilde{\cm}}[\dim \cn]$. In particular, for $G=SL_n(\C)$ we prove that each simple perverse sheaf that arises in this way pushes forward to a single Lusztig sheaf under the map $\mu$.  
    
The contents of the paper are as follows.
Section \ref{sec:preliminaries} covers background information and definitions. In Section \ref{sec.pushingforward}, we study the pushforward of the constant sheaf $\shq_{\tilde{\cm}}$ along $\tilde{\map}$.  As indicated above, some of our analysis in this section is carried out in a more general setting.  Section~\ref{sec: parabolic ic sheaves} studies the intersection cohomology complexes appearing in $\tilde{\map}_*\shq_{\tilde{\cm}}[\dim \cn]$.  We show that each of these pushes forward to an intersection cohomology complex on the partial Springer resolution $\tilde{\cn}^P$ for a particular parabolic subgroup $P$ of $SL_n(\C)$.  Finally, we realize each of these as an induced complex and prove our main result, Theorem~\ref{thm.genspringerA}, in Section~\ref{s.induction}.  We conclude with an explicit example for $G=SL_6(\C)$, and a discussion of future work.

\textit{Acknowledgments.} The authors are grateful to Pramod Achar for helpful conversations and feedback.  The second author was supported by an AWM-NSF travel grant during the course of this research.


\section{Preliminaries}\label{sec:preliminaries}

\subsection{Group actions and local systems}\label{sec.local-system-set-up} We work with schemes over the ground field $\C$.  Given a scheme $X$, $\underline{\QQ}_X$
denotes the constant sheaf on $X$ with rational coefficients.  The Lie algebra of an algebraic group is denoted
by the corresponding fraktur letter; if $G$ is an algebraic group, its identity component is denoted $G_0$,
and the component group is $G/G_0$.  

Suppose $Z$ is a finite group acting freely on $X$ on the right, and $V$ is a representation of $Z$.  There is a corresponding local system $\cl_V$ on $X/Z$, defined as the sheaf of locally constant sections of the vector bundle $\cv = X \times^Z V \to X/Z$.  By abuse of terminology, we may refer to $X \times^Z V$ as a
local system on $X/Z$.  

If $G$ is an algebraic group acting on $X$, $G^x$ denotes the stabilizer in $G$ of $x \in X$.  The component group $G^x/G^x_0$ of the
stabilizer is sometimes called the equivariant fundamental group of the orbit $G \cdot x$.   
Suppose $V$ is a representation
of $G^x/G^x_0$, viewed as a representation of $G^x$ via the map $G^x \to G^x/G^x_0$.  Then $V$ corresponds
to a local system $\cl_V$ on $G \cdot x$, defined as the sheaf of locally constant sections of the
vector bundle $G \times^{G^x} V \to G/G^x \cong G \cdot x$.

If $Z$ is a subgroup of $G$, and $M$ is a $Z$-variety, the mixed space $G \times^Z M$ is the quotient of $G \times M$ by the action of $Z$ defined by $(h,m) z = (hz, z^{-1} m)$ for $z\in Z$.  The equivalence class in
$G \times^Z M$ of $(g,m) \in G \times M$ is denoted $[g,m]$, or $[g,m]_Z$ if we wish to make the group explicit.

\subsection{Local systems on the principal orbit}\label{sec.local-system-nilp-orbit}  Throughout the rest of the paper, $G$ denotes a simply connected semisimple algebraic group over $\C$ with Lie algebra $\fg$ and center $Z$.  
The group of characters of $Z$ is denoted by $\widehat{Z}$ (with analogous notation for other finite groups).
Set $T_{ad}=T/Z$.   
We denote the set of nilpotent elements in $\fg$, known as the nilpotent cone, by $\cn$.  The group $G$ has a dense orbit
in $\cn$, called the principal nilpotent orbit. We denote this orbit by $\co^{pr}$ and refer to its elements as principal
nilpotent elements.  

Let $\nilp \in \fg$ be a principal nilpotent element. Choose a standard triple $\{ \nilp, s, \nilp_- \}$ with nilpositive element $\nilp$ and semisimple element $s\in \ft$. Since $\mathrm{ad}_s$ acts semisimply on $\fg$, we can decompose $\fg$ into its $\mathrm{ad}_s$-eigenspaces:
\[
\fg=\bigoplus_{j\in \Z} \fg_j,  \textup{ where } \fg_j =\{x\in \fg \mid [s, x]=jx\}.
\]
For any $k$, we let $\fg_{\ge k} = \bigoplus_{i \ge k} \fg_i$ and  $\fg_{> k} = \bigoplus_{i > k} \fg_i$.
Let $\ft = \fg_0$, $\fu = \fg_{>0}$, and $\fb = \ft + \fu$.  The corresponding subgroups of $G$ are $T$, $U$, and $B = TU$;
$B$ is a Borel subgroup of $G$ with maximal torus $T$.  

Let $\Phi$ denote the root system of $\fg$, with positive
roots $\Phi^+$ chosen so that $\fu$ is the sum of the positive root spaces.  The flag variety of $G$ is $G/B$, and $W=N_G(T)/T$ is the corresponding Weyl group.   
Much of this paper is focused on the case in which $G=SL_n(\C)$ and $\fg=\mathfrak{sl}_n(\C)$.  When $G=SL_n(\C)$ we may assume $B$ is the subgroup of upper triangular matrices; the Weyl group in this case is the symmetric group $S_n$.
We have $G^{\nilp} = Z U^{\nilp}$.  Since $U^{\nilp}$ is connected, being unipotent, the component group
$G^{\nilp}/G^{\nilp}_0$ is identified with $Z$.  Therefore any character $\chi \in \hat{Z}$ corresponds to a local
system $\cl_{\chi}$ on $G \cdot \nilp = \co^{pr}$.


\subsection{The Springer resolution} We denote the \textbf{Springer resolution} by $\mu: \tilde{\cn} \to \cn$ where
\[
\tilde{\cn} := \{(gB, x) \in G/B \times \cn \mid g^{-1}\cdot x \in \fu \}.
\]
Here $g^{-1}\cdot x$ denotes the adjoint action $Ad(g^{-1})(x) = g^{-1}xg$.  The map $\mu$ is simply projection onto the second factor. Given $x\in \cn$, the fiber
 \[
 \mu^{-1}(x)  =  \{(gB, x)  \mid g^{-1}\cdot x \in \fu\}
 \] 
 is the \textbf{Springer fiber} of $x$.   We can identify the Springer fiber with its image in $G/B$ under the projection to the
 first factor; under this identification, 
  \[
 \mu^{-1}(x)  =  \{gB \mid g^{-1}\cdot x \in \fu\}.
 \] 
 Sometimes $G/B$ is denoted by $\cb$ and the Springer fiber in $G/B$ is denoted $\cb^x$.

\begin{Rem} \label{remark: B-orbit description} There is an isomorphism of varieties
\[
G\times^B \fu \to \tilde{\cn},\; \;[g,x] \mapsto (gB, g\cdot x).
\]
(cf.~\cite[pg. 66]{Jantzen}).
We use this identification frequently below.  Under this identification,
$$
\mu: G\times^B \fu \to \cn, \; \; [g,x] \mapsto g\cdot x.
$$
Viewed as a subset of $G \times^B \fu$, the Springer fiber over $x$ is
\begin{equation} \label{e.SpringerBorbit}
 \mu^{-1}(x) = \{ [g, g^{-1} x] \mid g^{-1}\cdot x \in \fu\}.
\end{equation}
\end{Rem}


\subsection{The extended Springer resolution} \label{section: generalized resolution}  We now recall from
\cite{Gra:19} the definition of the variety which is the main geometric focus of this paper.   

Let $\Delta\subseteq \Phi^+$ denote the subset of simple roots.  Then $\fg_2=\bigoplus_{\alpha\in \Delta} \fg_{\alpha}$.
Since $[s,\nilp]=2\nilp$, we have $\nilp \in \fg_2$, and we can choose root vectors $E_{\ga}$ for each $\ga \in \Delta$
so that $\nilp = \sum_{\ga\in \Delta} E_{\ga}$.  Since the center $Z$ acts trivially on $\fg$, the action of $T$ on $\fg$ factors through
the map $T \to T_{ad} = T/Z$.   The map $T_{ad} \to T_{\ad} \cdot \nilp$ given by $t \mapsto t \cdot \nilp$ embeds 
$T_{ad}$ in $\fg_2$, so
 $\fg_2$ is an affine toric variety for $T_{ad}$, which we denote by 
\[
\toric_{ad}:=\fg_2 =\bigoplus_{\alpha\in \Delta}\fg_{\alpha}. 
\]
The composition $\toric_{ad} \to \fu \to \fu/[\fu, \fu]$ is an isomorphism, and using this,
we identify $\toric_{ad}$ with $\fu/[\fu,\fu]$.  Via this identification, $\toric_{ad}$
acquires a $B$-action (the subgroup $U$ acts trivially), and the projection
$p: \fu \to \fu/[\fu,\fu]=\toric_{ad}$ is $B$-equivariant.

An affine toric variety is characterized by the character group of the torus,
which can be viewed as a subset of the dual Lie algebra of the torus, together with a cone in the real span of 
the set of characters.  (The character group is a lattice in its real span.)  
The toric variety $\toric_{ad}$ corresponds
to the lattice given by the character group of $T_{ad}$, and the cone equal to the
set of $\R_{\geq 0}$-linear combinations of simple roots.  The simple roots can be viewed
as characters of either $T_{ad}$ or $T$.  We define $\toric$ to be the toric variety for $T$ 
obtained by changing the lattice for $\toric_{ad}$ but keeping the same cone: that is,
$\toric$ corresponds to the the lattice given by the character group of $T$, and the cone equal to the
set of $\R_{\geq 0}$-linear combinations of simple roots.  
It follows that $\toric/Z = \toric_{ad}$ (see \cite[Section 2.2]{Fulton}). 

The composition $T \to B \to B/U$ is an isomorphism,
so identifying $T$ with $B/U$, there is a natural projection $B \to T$.  Via this projection,
$\toric$ acquires a $B$-action (where $U$ acts trivially), and the projection $\pi:\toric
\to \toric_{ad}$ is $B$-equivariant.

The variety $\tilde{\cm}$ and map $\psi: \tilde{\cm} \to \cn$ discussed in the introduction are defined as follows.  First, consider the maps $p: \fu \to \fu/[\fu,\fu]=\toric_{ad}$ and $\pi: \toric \to \toric/Z=\toric_{ad}$ defined in the paragraphs above.  Set
\[
\tilde{\fu} := \toric\times_{\toric_{ad}} \fu =\{ (v, y)\mid \pi(v)=p(y) \}
\] 
i.e., we form the following Cartesian diagram.
\[
\xymatrix{\tilde{\fu} \ar[r] \ar[d] & \fu \ar[d]^p  \\ \toric  \ar[r]^{\pi}& \toric_{ad} }
\]
Because the maps $p$ and $\pi$ are both $B$-equivariant, $B$ acts on $\tilde{\fu}$.  We define $\tilde{\cm}:=G\times^B\tilde{\fu}$.  Let $\map: \tilde{\fu} \to \fu$ denote  projection onto the second factor.  We then define $\tilde{\map}: \tilde{\cm} \to \tilde{\cn}$ to be the map induced from $\map$, so $\tilde{\map}$ maps the element $[g,x]\in \tilde{\cm}$ to $[g, \map(x)]\in \tilde{\cn}$.  The \textbf{extended Springer resolution} is the variety $\tilde{\cm}$, together with the map $\psi: \tilde{\cm}\to \cn$, where $\psi$ is the composition
\[
\xymatrix{ \psi: \tilde{\cm} \ar[r]^{\tilde{\map}} & \tilde{\cn} \ar[r]^{\mu} & \cn} 
\]
of $\tilde{\map}$ with the usual Springer resolution $\mu$.


\subsection{Intersection cohomology sheaves and the decomposition theorem} \label{sec.decomp}

The decomposition theorem 
plays an important role in this paper, and we briefly discuss and state it here. The original version of this
theorem is due to Beilinson, Bernstein, Deligne, and Gabber \cite{BBD}; we will also need a generalization
which can be found in  \cite{PCMI} or \cite{deCataldo}.  Statements and discussions of the theorem relevant for this paper can also be found in \cite{DeCataldoMigliorini}, \cite{Borho1983}, and \cite{Jantzen}.  Before we state the decomposition theorem, we first discuss its essential ingredients, namely intersection cohomology complexes.  These are the simple perverse sheaves, and objects in the bounded constructible derived category $D^b(X)$ of a variety $X$.  

A perverse sheaf is a constructible complex of sheaves of $\QQ$-vector spaces which satisfies certain support and co-support conditions.  The intersection cohomology complexes are determined up to canonical isomorphism in $D^b(X)$ by more restrictive conditions.  Suppose $X$ is a complex algebraic variety, and let $U$ be a nonsingular, dense open subvariety of $X$.  Given a local system $\cl$ on $U$, and an integer $d$, we let $\cl[d]$ be the complex in $D^b(X)$ with $\cl$ shifted to the $-d$th location.  The \textbf{intersection cohomology complex} $IC(X,\cl)$ is the object in $D^b(X)$ uniquely determined up to canonical isomorphism by the following properties: the restriction to $U$ is $\mathcal L[\dim X]$,
\[
\dim\{x\in X \mid \mathcal H^i_x(IC(X,\cl))\neq 0\}<-i\text{ for all } i>\dim X
\] 
and 
\[
\dim\{x\in X \mid \mathcal H^i_x(\mathbb D IC(X,\cl))\neq 0\}<-i\text{ for all  } i>-\dim X
\] 
where $\ch^i_x$ denotes the stalk at $x$ of the cohomology sheaf of the complex, and $\mathbb D$ denotes the Verdier dual.  If $X$ is smooth and $\cl_{triv}$ is the trivial local system, $IC(X, \cl_{triv}) =\underline{\mathbb{Q}}_X[\dim X]$ (recall that $\shq_{X}$ denotes the constant sheaf on $X$).  Spaces with the property that $IC(X,\cl_{triv}) = \underline{\mathbb Q}_X[\dim X]$ are called \textbf{rationally smooth}.  Examples of rationally smooth varieties include $\widetilde{\cm}$ (which is rationally smooth since it is locally a quotient of a smooth
variety by a finite group) and $\cn$ (see \cite[\S 2.3]{Borho1983}).

Suppose $j: Z \hookrightarrow X$ is the inclusion of a closed subvariety.  
Since we are only working in the derived setting, we use the notation $j_*$ rather than $Rj_*$ to denote the derived direct image $D^b(Z) \to D^b(X)$.  The functor $j_*$ is fully faithful and induces an equivalence of 
categories between $D^b(Z)$ and the set of objects in $D^b(X)$ supported on $Z$ \cite[Section 2.3]{DeCataldoMigliorini}.  
Moreover, for any $\cf \in D^b(Z)$, 
there is a natural isomorphism $j^* j_* \cf \to \cf$ (cf.~\cite[p.~102]{Iversen}).  
By convention, if we are working with $D^b(X)$, we omit the
symbol $j_*$, and write $IC(Z,\cl)$ for $j_* IC(Z,\cl)$, where $\cl$ is a local system on
an open subset of $Z$.  More generally, if $S$ is a closed subvariety of $Z$ and $\cl$ is a local system on
an open subset of $X$, $IC(S,\cl)$ may denote an element of $D^b(Z)$ or $D^b(X)$.  There
is a natural isomorphism $j^* IC(S,\cl) \to IC(S,\cl)$, where $IC(S,\cl)$
is viewed on the left hand side as an element of $D^b(Z)$ and on the right hand side as
an element of $D^b(X)$.


All simple perverse sheaves in $D^b(X)$ are of the form $IC(S,\mathcal L)$ for some closed
subvariety $S$ and local system $\mathcal L$.  Note that if $S$ has two different open dense subvarieties $U_1$ and $U_2$, and local systems $\cl_1$ on $U_1$ and $\cl_2$ on $U_2$ that agree on the intersection $U_1\cap U_2$, then $IC(S,\cl_1)$ and  $IC(S,\cl_2)$ are canonically isomorphic.

Let $f:X\rightarrow Y$ be an algebraic map of irreducible complex varieties.  Then $f$ is \textbf{semismall} if for all $d$, 
\[
\dim\{p\in X\mid \dim f^{-1} (p) \geq d\}\leq \dim X - 2d.
\] 
If $f$ is semismall, then $f$ is generically finite.
The map $f$ is \textbf{small} if the inequality above is strict for $d>0$.  A finite map is small.

We now state a version of the decomposition theorem for perverse sheaves. The first statement here can be found in \cite{PCMI} or \cite{deCataldo}.  The second statement comes from \cite{Borho1983} and is a specialization of the original version found in \cite{BBD}.  

\begin{Thm}[The Decomposition Theorem]\label{decomp}
Let $f:X\rightarrow Y$ be a proper map of complex algebraic varieties.  
\begin{enumerate}
\item Let $\mathcal L$ be a semisimple local system on $X$.  Then the derived pushforward $f_*IC(X,\mathcal L)$ is a finite direct sum of shifted simple perverse sheaves on $Y$.  
\item If $f$ is semismall and $X$ is rationally smooth, then $f_*\underline{\mathbb Q}_X[\dim X]$ is perverse, so no nontrivial shifts occur in (1).
That is, 
\begin{eqnarray}\label{eqn.decompthm}
f_*\underline{\mathbb Q}_X[\dim X] \cong \bigoplus_i IC(\overline{S_i},\mathcal{L}_i)\otimes V_i
\end{eqnarray} 
where each $S_i$ is a locally closed subvariety of $Y$, $\mathcal L_i$ is an irreducible local system on an open set of $S_i$ and $V_i$ is a vector space with dimension equal to the multiplicity of  $IC(\overline{S_i},\mathcal{L}_i)$ in the sum. 
\end{enumerate}
\end{Thm}


\subsection{The Springer Correspondence and Lusztig's Generalized Springer Correspondence}\label{sec.genspringer}

The Springer correspondence is an injective map from the set of isomorphism classes of irreducible $W$-representations to pairs $(\co, \cl)$, where $\co$ is a nilpotent orbit and $\cl$ is an irreducible local system on $\co$.  Note that
each pair $(\co, \cl)$ corresponds to a simple perverse sheaf $IC(\overline{\co},\cl)$ on $\cn$.  In \cite{Borho1983}, Borho and MacPherson give a proof of the Springer correspondence which relies on the theory of perverse sheaves.  They apply Theorem~\ref{decomp} to the Springer resolution to obtain
\begin{equation} \label{e.springercorr}
\mu_*\underline{\QQ}_{\widetilde{\cn}}[\dim \cn] = \bigoplus_{(\co,\cl)} IC(\overline{\co},\cl)\otimes V_{(\co,\cl)}.
\end{equation}
The left side of~\eqref{e.springercorr} is called the \textbf{Springer sheaf}; it is the derived pushforward of the constant sheaf on $\widetilde{\cn}$ along the Springer resolution.  On the right side of~\eqref{e.springercorr}, the sum is over the pairs $(\co, \cl)$ that appear in the Springer correspondence.  One can show that the Weyl group acts on the left side of~\eqref{e.springercorr}.  
This implies that each $V_{(\co,\cl)}$ is a $W$-representation; it is precisely the irreducible representation of $W$ mapping to $(\co,\cl)$ under the Springer correspondence.

Since the Springer correspondence is an injection, we recover only some of the pairs $(\co, \cl)$ in this manner.  Those missing from the Springer correspondence were described by Lusztig in all cases \cite{Lu84}.  Following Lusztig, we consider triples $\textbf{c} = (L,\co_{L}, \mathcal{L}^L)$, where $L$ is a standard
Levi subgroup of $G$, $\co_{L}$ is a nilpotent orbit of $L$, and $\mathcal{L}^L$ is a local
system on $\co_{L}$.  A triple satisfying certain conditions is called a cuspidal datum, and
Lusztig has classified all such triples.  
The \textbf{relative Weyl group} corresponding to a cuspidal datum $\textbf{c}$ is
the group $W(L):=N_G(L)/L$.  Although $W(L)$ is defined even if $\textbf{c}$ is not cuspidal,
the cuspidal hypothesis is required to ensure that $W(L)$ is a Coxeter group. 
Details can be found in \cite{Lu84} and \cite{LuSpal85}.  The \textbf{generalized Springer correspondence} is a bijection between the set of simple perverse sheaves on the nilpotent cone and irreducible representations of relative Weyl groups associated to cuspidal data.

Suppose $P$ is a parabolic subgroup of $G$ with Levi decomposition $P=LU_P$, where $L$ is part of a cuspidal datum $\textbf{c}$ for $G$.  The corresponding decomposition of Lie algebras is $\fp = \fl + \fu_P$.  By means of the isomorphism $\fl \simeq \fp/\fu_P$, the Lie algebra $\fl$ acquires a $P$-action, where $U_P$
acts trivially.  The projection $q: \fl+\fu_P \to \fl$ is $P$-equivariant, and takes $\cn_L + \fu_P$ to $\cn_L$.  Note that
$U_P$ acts trivially on $\cn_L$ when $\cn_L$ is viewed as a subvariety of $\fp/\fu_P$, but $\cn_L$ is not $U_P$-invariant
when $\cn_L$ is viewed as a subvariety of~$\cn_L + \fu_P$.

Let $\tilde{\cn}^P=G\times^P(\cn_L+\fu_P)$ denote the variety considered by Borho and MacPherson in~\cite{Borho1983}; here $\cn_L$ is the nilpotent cone of $L$.  We write $[g,x]_P$ for the equivalence class
of $(g,x)$ in $\tilde{\cn}^P$, where $g \in G$ and $x \in \cn_L+\fu_P$.
Borho and MacPherson call the map $\mu_P: \tilde{\cn}^P \to \cn$ defined by $\mu_P([g,x]_P) = g\cdot x$ a partial resolution of
$\widetilde{\cn}$; it is analogous to the Springer resolution.  Note that using an isomorphism analogous to the one defined in Remark~\ref{remark: B-orbit description},  $\tilde{\cn}^P$ can be identified with the variety of pairs $(gP, x)\in G/P\times \cn$ such that~$g^{-1}\cdot x \in \cn_L +\fu_P$. 

Let $\pi_P: \tilde{\cn}^P \to G\times^P \cn_L$ be the map induced from $q: \fl+\fu_P \to \fl$.
We have the following diagram:
\[\cn\xleftarrow{\;\mu_P\;} \widetilde{\cn}^P:=G\times^P (\cn_L+\fu_P)\xrightarrow{\;\pi_P\;} G\times^P \cn_L{\longrightarrow} \cn_L.\]

By the induction of complexes defined in \cite{Lusztig1985} (see, in particular, \cite[(1.9.3)]{Lusztig1985} and \cite[2.6.3]{BLEquiv}), each $L$-equivariant object in the derived category $D^b(\cn_L)$ determines a $G$-equivariant object in the derived category $D^b(G\times^P \cn_L)$.  We denote the corresponding functor by $\Ind_P^G$. Consider the simple perverse sheaf $\Ind_P^GIC(\overline{\co}_{L},\cl^L)$ on $G\times^P\cn_L$, where $\co_{L}$ and $\cl^L$ are from the cuspidal datum $\textbf{c}=(L,\co_L, \cl^L)$.   We define
\begin{eqnarray}\label{e.lusztigsheaf}
\mathbb A_\textbf{c} := {\mu_P}_*{\pi_P}^*\Ind_P^GIC(\overline{\co}_{L},\cl^L)[\dim\fu_P]
\end{eqnarray}
to be the \textbf{Lusztig sheaf} associated to~$\textbf{c}$.  

The Lusztig sheaf $\mathbb{A}_\textbf{c}$ decomposes as a sum of simple perverse sheaves, each of which corresponds to an irreducible representation of $W(L)$.  The generalized
Springer correspondence is obtained by considering the Lusztig sheaves $\mathbb{A}_\textbf{c}$
as
 $\textbf{c}$ varies over all cuspidal data.  When $L=T$, the only cuspidal datum is $\textbf{c} = (T, \co_{ \{0\} }, \cl_{triv}^T)$ where $\co_{\{0\}}$ is the zero orbit.  In this
case, the relative Weyl group is the full Weyl group $W$ and the Lusztig sheaf
is the Springer sheaf.  We therefore recover the Springer correspondence as a special case of the generalized Springer correspondence.

\begin{Rem} Lusztig's generalized Springer correspondence is proven for connected reductive algebraic groups over any algebraically closed field (of possibly positive characteristic $p$) and $\overline{\mathbb Q}_\ell$-sheaves.  Although the discussion in this paper focuses on the characteristic zero setting and $\mathbb Q$-sheaves, most of the results generalize to the positive characteristic setting.  The only potential difficulty is our analysis of varieties with a finite group action (e.g.~the $Z$-action on $\tilde{\cm}$) in Section~\ref{sec.pushingforward}.  However, if $X$ is a variety defined over an algebraically closed field of characteristic $p$ and $Z$ is a finite abelian group acting on $X$, the results of Section~\ref{sec.pushingforward} hold whenever $p$ is relatively prime to  $|Z|$.  
\end{Rem}


\subsection{Cuspidal data and Lusztig sheaves for $G=SL_n(\C)$}\label{sec.typeA}

When $G=SL_n(\C)$, the Springer correspondence and its generalization have a more direct description.  In this setting, the only pairs $(\co,\cl)$ appearing on the right hand side of~\eqref{e.springercorr} are those corresponding to the trivial local system $\cl_{triv}$ on each orbit $\co$, thus the Springer correspondence gives a bijection between irreducible representations of $S_n$ and nilpotent orbits.  We adopt the conventions of~\cite{Jantzen}, so the irreducible $S_n$-representation $V_\lambda$ indexed by the partition $\lambda\vdash n$ corresponds to the nilpotent orbit $\co_\lambda$ of nilpotent matrices with Jordan type $\lambda$.  Note that our conventions may differ from some others appearing in the literature up to tensoring $V_\lambda$ with the sign representation.

Let $\nilp\in \co_\lambda$.  Each irreducible local system on the orbit $\co_\lambda$ corresponds to a unique irreducible representation of the equivariant fundamental group $G^{\nilp}/G_0^{\nilp}$, as described in Section~\ref{sec.local-system-set-up} above (in fact, this correspondence is an equivalence of categories). If $\lambda=(\lambda_1, \lambda_2, \ldots, \lambda_k)$, we have $G^{\nilp}/G_0^{\nilp} \simeq \Z_m$ where $m=\gcd(\lambda_1,\lambda_2, \ldots, \lambda_k)$.  The center $Z$ of $G$ is isomorphic
to the cyclic group $\Z_n$.  Since $Z$ is a subgroup of $G^\nilp$, we obtain a map $Z\to G^\nilp/G^\nilp_0$.  This map is always surjective in the type $A$ setting, so we obtain an injective map $\widehat{G^\nilp/G^\nilp_0} \to
\widehat{Z}$ of the character groups.  Suppose  $\chi \in \widehat{Z}$ is in the image of this map; note
that the order of $\chi$ divides $m$.  We denote the corresponding $G$-equivariant local system on $\co_\lambda$  by 
$\cl_\chi$; it depends on 
the orbit $\co_\lambda$ as well as on $\chi$, but we omit $\lambda$ from the notation.  Recall that if $\lambda=(n)$, then $\co_\lambda=\co^{pr}$ is the principal orbit and $G^{\nilp}/G_0^{\nilp} \simeq Z$.  

From \cite{Lu84} and \cite{LuSpal85}, it follows that there is a character of $Z$ associated to each Lusztig sheaf $\mathbb A_\textbf{c}$.  In the type $A$ setting, this association is a bijection.  Given $\chi\in \widehat{Z}$, we denote
by  $\textbf{c}_\chi$ the unique cuspidal datum associated to $\chi$, and $\mathbb A_\chi$ the Lusztig sheaf constructed from $\textbf{c}_\chi$.  If  $\chi$ has order $d$, then $\textbf{c}_\chi = (L,\co_{L}^{pr}, \mathcal{L}_\chi^L)$, where the Levi subgroup $L$ is of the form $S(GL_d(\mathbb C)\times \cdots \times GL_d(\mathbb C))$ with $n/d$ factors, $\co_{L}^{pr}$ is the principal orbit in $\cn_L$, and $\cl_\chi^L$ is a local system defined precisely in Section~\ref{s.induction} below (see also \cite[\S 5]{LuSpal85}).  The relative Weyl group in this case is isomorphic to $S_{n/d}$.   

Given a character $\chi$ of order $d$, we obtain an induced local system $\cl_\chi$ on precisely those orbits $\co_\lambda$ for which $\lambda=(\lambda_1,\ldots, \lambda_k)$ is a partition of $n$ such that $d$ divides $\lambda_i$ for all $i$.  Let $\bar\lambda$ denote the partition of $n/d$ obtained by dividing each part of $\lambda$ by $d$.  The Lusztig sheaf $\mathbb{A}_\chi$ corresponding to $\chi$ decomposes as:
\begin{equation} \label{e.decompLusztig}
\mathbb{A}_\chi = \bigoplus_{\substack{\lambda=(\lambda_1, \ldots, \lambda_k) \vdash n\\  d\vert \lambda_1, \ldots, d\vert\lambda_k}} IC(\overline{\co}_\lambda, \cl_\chi)\otimes V_{\bar\lambda}.
\end{equation}
The generalized Springer correspondence can be stated explicitly in this type A setting as the bijection between simple perverse sheaves on the nilpotent cone and irreducible representations of the relative Weyl groups given by $IC(\bar{\co}_\lambda, \cl_\chi)\mapsto V_{\bar\lambda}$.  The following example computes the decomposition \eqref{e.decompLusztig} for $G=SL_4(\C)$.

\begin{example}\label{ex: Lusztig sheaves n=4}  The group $G=SL_4(\C)$ has center $Z =\{\omega^k I_4\}$, where $\omega=\exp(\pi i/2)$ and $I_4$ is the $4\times 4$ identity matrix.  Write $\widehat{Z} = \{\iota, \chi_1, \chi_2, \chi_3 \}$, where $\chi_k$ denotes the central character such that $\chi_k(\omega I_4) = \omega^k $, and $\iota$ is the trivial character.  The Lusztig sheaf corresponding to the trivial character is the Springer sheaf $\mu_*\shq_{\tilde{\cn}}$. Equation~\ref{e.springercorr} becomes
\[
\mu_*\shq_{\tilde{\cn}}[\dim\cn] = \bigoplus_{\lambda\vdash 4} IC(\overline{\co}_\lambda, \cl_{\iota})\otimes V_{\lambda}.
\]
 
The Lusztig sheaf corresponding to $\chi_2$ is
\[
\A_{\chi_2} = \left( IC(\overline{\co}_{[4]}, \cl_{\chi_2}) \otimes V_{[2]} \right) \oplus  \left( IC(\overline{\co}_{[2,2]}, \cl_{\chi_2} ) \otimes V_{[1,1]}\right),
\]
where the vector spaces $V_\lambda$ for $\lambda\vdash 2$ are representations of $W(L) \simeq S_2$.
Similarly, the Lusztig sheaves corresponding to $\chi_1$ and $\chi_3$ are:
\[
\A_{\chi_1} = IC(\overline{\co}_{[4]}, \cl_{\chi_1}) \otimes V_{[1]} \;\textup{ and } \; \A_{\chi_3} = IC(\overline{\co}_{[4]}, \cl_{\chi_3}) \otimes V_{[1]}
\] where $V_{[1]}$ corresponds to the only representation of the trivial group $W(L)\simeq S_1$ in both these cases.
Since $\dim(V_{[1,1]}) = \dim(V_{[2]}) = 1$ and $\dim(V_{[1]}) =1 $, we see that the only simple perverse sheaves with multiplicity more than one in this case are those coming from the Springer sheaf.
\end{example}

\vspace*{.05in}

In the next sections, we apply the decomposition theorem to the extended Springer resolution when $G=SL_n(\C)$.  We will show that we recover \textit{all} of the Lusztig sheaves $\mathbb A_\chi$ as summands of the pushforward of the constant sheaf.  That is, we recover all the simple perverse sheaves on the nilpotent cone.  Moreover, each occurs with multiplicity given by the dimension of the corresponding irreducible representation of $S_{n/d}$.


\section{Pushing forward the constant sheaf to $\tilde{\mathcal N}$} \label{sec.pushingforward}
The main results of the next few sections concern the pushforward of the constant sheaf $\shq_{\tilde{\cm}}$ via
the composition
$$
\psi: \tilde{\cm} \xrightarrow{\;\tilde{\map}\;}\tilde{\cn}  \xrightarrow{\;\mu\;}  \cn.
$$
In this section, we study the pushforward under the first map
$\tilde{\map}: \tilde{\cm} \to \tilde{\cn}$. We prove some of the needed results in a more general setting.  The main general result is Proposition~\ref{prop.pushforward}.

\subsection{Pushing forward along a finite quotient map} \label{s.local}
Throughout this subsection, $X$ will denote an $n$-dimensional complex
 variety with a left action of a finite abelian group $Z$.  It will frequently be convenient
 to use the right $Z$-action on $X$ defined by
 the formula $xz := z^{-1} x$; we write $Y=X/Z$ for the quotient of $X$ by $Z$ (with either action).
 We will assume that there is an open set $X_{reg} \subset X$ such that $Z$ acts freely on $X_{reg}$, and let $Y_{reg} = X_{reg}/Z$.  We have the following commutative
 diagram, where the horizontal maps are open embeddings, and the vertical maps are quotient maps.
 $$
 \xymatrix{
 X_{reg} \ar[r]^i \ar[d]_\rho& X \ar[d]^\pi \\
 Y_{reg} \ar[r]^j&  Y
 }
 $$
 The maps $\pi$ and $\rho$ are finite, so they are trivially both small and semi-small.
 Moreover, the map $\rho$ is \'etale.
The fibers of $\pi$ are exactly the orbits of $Z$ on $X$.  

Let $\chi: Z \to \C^*$ denote a character of $Z$, and let $\cl_{\chi}$ denote the local
system defined, as in Section~\ref{sec.local-system-set-up}, by the sheaf of locally constant sections of the line bundle $\cv_{\chi}=X_{reg} \times^Z \C_{\chi}$ over 
$X_{reg}/Z = Y_{reg}$.  These local systems are studied in \cite[Section 3.2]{AHJR2}; the following lemma is essentially a special
case of \cite[Lemma 3.6]{AHJR2}.

\begin{Lem} \label{lem.pushfree}
With notation as above, we have
$$
\rho_* \shq_{X_{reg}} = \bigoplus_{\chi \in \widehat{Z}}  \cl_{\chi}.
$$
\end{Lem}

For $\chi \in  \widehat{Z}$, define
\[
X_\chi = \{x\in X \mid Z^x \subseteq \ker(\chi)\}.
\]
Since $Z$ is abelian, the set $X_{\chi}$ is $Z$-stable; let $Y_{\chi} = X_{\chi}/Z = \pi(X_{\chi})$.  Because
$X_{\chi}$ is $Z$-stable, if $y \in Y_{\chi}$,
the entire preimage $\pi^{-1}(y)$ lies in $X_{\chi}$.  In fact, since $Z$ is abelian and
$\pi^{-1}(y)$ is a single $Z$-orbit, the stabilizer groups $Z^x$ are the same for all $x \in \pi^{-1}(y)$.
 
Given $\chi \in \widehat{Z}$, write $H_{\chi} = \ker \chi$.  If $\chi$ is understood, we
write $H = H_{\chi}$.
We can view $\chi$ either
as a character of $Z$ or of $Z/H$.

\begin{Lem} \label{lem.open}
The set $X_{\chi}$ is open in $X$.  Hence $Y_{\chi}$ is open
in $Y$.
\end{Lem}

\begin{proof}
By definition,
$x \in X \setminus X_{\chi}$ if and only if there exists
$g \in Z \setminus H$ such that $x \in X^g=\{x\in X\mid gx=x\}$.  Thus,
$$
X \setminus X_{\chi} = \cup_{g \in Z \setminus H} X^g.
$$
Since each $X^g$ is closed and $Z$ is finite, 
$X \setminus X_{\chi}$ is closed and thus $X_{\chi}$ is open in $X$.
Finally, $\pi: X \to Y$ is proper, so $\pi(X \setminus X_{\chi})
= Y \setminus Y_{\chi}$ is closed in $Y$.  Thus, $Y_{\chi}$ is open
in~$Y$.
\end{proof}

\begin{Lem}
The group $Z/H$ acts freely on $X_{\chi}/H$.
\end{Lem}

\begin{proof} 
Let $\bar{z}, \bar{x}$ denote the images of $x\in Z$ and $x\in X_\chi$ in $Z/H$ and $X_\chi/H$, respectively.
Suppose $z\in Z$ and $x\in X_\chi$ such that $\bar{z} \bar{x} = \bar{x}$.  We must show
that $\bar{z} = \bar{e}$.  By definition, $\bar{z} \bar{x} = \bar{x}$ implies $z x = h x$ for some $h \in H$.  Thus,
$z h^{-1} x = x$, implying $z h^{-1} \in Z^x \subset H$.
Hence $z \in H$, so $\bar{z} = \bar{e}$, as desired.
\end{proof}

By definition, the open set $X_{\chi}$ contains $X_{reg}$ and hence $Y_{\chi}$ contains
$Y_{reg}$.   

\begin{Lem}\label{lemma: Z/H action} We have
$$
\cv_{\chi} = X_{reg}/H \times^{Z/H} \C_{\chi}.
$$
\end{Lem}

\begin{proof}
By definition, $\cv_{\chi} = X_{reg} \times^Z \C_{\chi}$ equals $(X_{reg} \times \C_{\chi})/Z$,
where $Z$ acts by the mixing action $z(u,c) = (u z^{-1}, zc)$.
We have
\begin{equation} \label{e.finitequotient}
(X_{reg} \times \C_{\chi})/Z =  \big( (X_{reg} \times \C_{\chi})/H \big) / (Z/H).
\end{equation}
Since $H$ acts trivially on $\C_{\chi}$, the right hand side
of~\eqref{e.finitequotient} is equal to
$( (X_{reg}/H) \times \C_{\chi})/(Z/H)$, which by definition
equals $X_{reg}/H \times^{Z/H} \C_{\chi}$.
\end{proof}

\begin{Prop}\label{prop.extension} The local system $\cl_{\chi}$ on $Y_{reg} = X_{reg}/Z$
extends to a local system \textup{(}also denoted by $\cl_{\chi}$\textup{)}
on $Y_{\chi}$.
\end{Prop}

\begin{proof}
It is equivalent to show that there is a local system
on $Y_{\chi}$ whose restriction to $Y_{reg}$ is $\cl_{\chi}$. 
Since $Z/H$ acts freely on $X_{\chi}/H$, we have a local system associated to the line bundle
$X_{\chi}/H \times^{Z/H} \C_{\chi}$ on $(X_{\chi}/H) / (Z/H) = X_{\chi}/Z$.
By Lemma~\ref{lemma: Z/H action}, the restriction of this local system 
to $X_{reg}/H$ is equal to $\cl_{\chi}$.
\end{proof}

Recall that $n = \dim X = \dim Y$.
As noted above, given $y \in Y$, the stabilizer group $Z^x$ is independent of the choice of $x \in \pi^{-1}(y)$; a choice of such
$x$ gives an identification of $\pi^{-1}(y)$ with $Z/Z^x$.
Given a complex $\cf$ of sheaves on $Y$, the cohomology sheaf
$\ch^i(\cf)$ has stalk at $y \in Y$ denoted by
$\ch^i_y (\cf)$.  
We have $\dim \ch^i_y(\pi_* \shq_X) = \dim H^i(\pi^{-1}(y))$.
Because $\pi$ is a finite map, this dimension is 
$0$ if $i>0$, and is equal to $|\pi^{-1}(y)| = |Z/Z^x|$
if $i=0$.

The next result describes the pushforward of the constant sheaf on $X$ under the quotient map $\pi:X \to Y$.
The strategy of the proof is to first use the decomposition theorem and Lemma
\ref{lem.pushfree} to show that $\pi_* \shq_X[n]$ equals the expression in the statement of the proposition plus possible extra terms, and then to use dimension arguments to show that these extra terms do not appear.

\begin{Prop} \label{prop.pushforward} Let $\pi: X \to Y$ be the quotient of an
$n$-dimensional variety $X$ by the action of a finite abelian group $Z$, such that $Z$ acts freely on an open subset $X_{reg}$ of $X$.
Assume that $X$ is rationally smooth.  With notation as above, we have
\begin{enumerate}
\item $\pi_* \shq_X[n] = \bigoplus_{\chi \in \widehat{Z}} IC(Y, \cl_{\chi})$, and
\item $\dim \ch^{-n}_y( IC(Y, \cl_{\chi}))$ is equal to 
$1$ if $y \in Y_{\chi}$ and is $0$ otherwise.
\item If $y \not\in Y_{\chi}$, then 
\begin{equation} \label{e.stalkzero}
\ch^{k}_y (IC(Y, \cl_{\chi})) = 0 \mbox{ for all integers }k.
\end{equation}
  Hence, if $i: Y \setminus Y_{\chi}\hookrightarrow Y$
is the inclusion, we have $i^* IC(Y, \cl_{\chi}) = 0$
in $D^b(Y \setminus Y_{\chi})$.

\item If $j_{\chi}: Y_{\chi} \to Y$ is the inclusion, then $IC(Y, \cl_{\chi}) =(j_{\chi})_! \cl_{\chi}[n]$.
\end{enumerate}

\end{Prop}

\begin{proof}
Because $X$ is rationally smooth, $\shq_X[n]$ is a perverse sheaf.  Since $\pi$ is a small
map, $\pi_* \shq_X[n]$ is perverse as well.  By the Decomposition Theorem (Theorem~\ref{decomp} above),
 \begin{equation} \label{e.pushforwardgeneral}
  \pi_* \shq_X[n] = \bigoplus_i IC(Y_i, \cl_i), 
 \end{equation}
where $Y_i$ is a closed subvariety of $Y$ and $\cl_i$ is a local system on
an open subvariety of $Y_i$.
We claim
 that for each $\chi \in \widehat{Z}$, the sheaf $IC(Y, \cl_{\chi})$ occurs in the sum, so
 the sum takes the form
\begin{equation} \label{e.pushforward}
\pi_* \shq_X[n] = \left(\bigoplus_{\chi \in \widehat{Z}} IC(Y, \cl_{\chi})\right)
 \oplus \ck,
 \end{equation}
where $\ck$ is a sum of $IC$ complexes
 for local systems on subvarieties of $Y$.  
To verify the claim, recall that by Lemma \ref{lem.pushfree},
$$
\rho_* \shq_{X_{reg}}[n] = \bigoplus_{\chi \in \widehat{Z}}  \cl_{\chi}[n].
$$
Each term in the decomposition of $\pi_* \shq_X[n]$ is of the form $IC(Y_i, \cl_i)$, where $Y_i$ is a closed
subvariety of $Y$.  Recall that $j: Y_{reg} \hookrightarrow Y$.
By base change, $j^* \pi_* = \rho_* i^*$. Hence, the
restriction of the right side of~\eqref{e.pushforwardgeneral} to $Y_{reg}$
equals $\bigoplus_{\chi \in \widehat{Z}}  \cl_{\chi}[n]$.  Therefore, for each $\chi$,
there exists an $i$ such that the restriction of $IC(Y_i, \cl_i)$ to $Y_{reg}$ is
$\cl_{\chi}[n]$.  To finish proving the claim, it suffices to verify that
the only intersection cohomology complex on
$Y$ whose restriction to $Y_{reg}$ is $\cl_{\chi}[n]$ is $IC(Y, \cl_{\chi})$.  To see this, suppose
that the restriction of 
$IC(Y_i, \cl_i)$ to $Y_{reg}$ equals $\cl_{\chi}[n]$.  Since the support of $IC(Y_i, \cl_i)$
is $Y_i$, we must have $Y_i = Y$, and then $\cl_i$ is a local system on an open subset $V_i$ of
$Y$.  This implies that the restrictions of $\cl_{\chi}$ and $\cl_i$ to $V_i \cap Y_{reg}$
agree.  But if two local systems on open sets agree on the intersection of those open sets,
then the corresponding $IC$-complexes are canonically isomorphic.  We conclude that $IC(Y_i, \cl_i) =
IC(Y, \cl_{\chi})$, as desired.  This proves the claim.

Now let $y \in Y$ and $x \in \pi^{-1}(y)$.  
By the remarks preceding the proposition,
 \begin{equation} \label{e.dimension}
\dim \ch^{-n}_y(\pi_* \shq_X[n]) =  \dim \ch^0_y(\pi_* \shq_X) = |\pi^{-1}(y)|  = |Z/Z^x|,
 \end{equation}
 and 
\begin{equation} \label{e.dimension2}
 \dim \ch^{k}_y(\pi_* \shq_X[n]) = 0 \mbox{ if } k \neq -n.
 \end{equation}
  
Combining \eqref{e.pushforward}, \eqref{e.dimension}, and \eqref{e.dimension2}, we see that for any $y \in Y$,
\begin{equation} \label{e.pushforward1}
|Z/Z^x| = \dim \ch^{-n}_y(\pi_* \shq_X[n]) = \sum_{\chi \in \widehat{Z}} \dim \ch^{-n}_y(IC(Y, \cl_{\chi}))
 + \dim H^{-n}_y(\ck)
\end{equation} 
and 
\begin{equation} \label{e.cohzero}
\dim \ch^{k}_y(IC(Y, \cl_{\chi})) = 
\dim \ch^{k}_y(\ck) = 0 \mbox{ if }k \neq -n.
\end{equation}

 By Proposition~\ref{prop.extension}, $\cl_{\chi}$ extends to a local system (again denoted
 by $\cl_{\chi}$) on $Y_{\chi}$ so $IC(Y, \cl_{\chi})|_{Y_{\chi}} = \cl_{\chi}[n]$.  Therefore if $y \in Y_{\chi}$, then  $\dim \ch^{-n}_y (IC(Y, \cl_{\chi})) = 1$.  
We claim that the number of characters $\chi\in \widehat{Z}$ such
that $y \in Y_{\chi}$ is $|Z/Z^x|$.
Indeed, $y\in Y_{\chi}$
 if and only if $Z^x$ is contained in $\ker \chi$.
 This is equivalent to saying that $\chi$ is pulled
 back from a character $\bar{\chi}$ of $Z/Z^x$
 by the quotient map $Z \to Z/Z^x$.  Thus, the
 number of $\chi \in \widehat{Z}$ such that $y \in Y_{\chi}$
 is equal to the number of distinct characters of
 $Z/Z^x$, which is equal to $|Z/Z^x|$ since $Z/Z^x$
 is an abelian group.  This proves the claim.
 
The discussion in the previous paragraphs shows that the contribution to the sum on the
right hand side of \eqref{e.pushforward1} from the characters
$\chi$ such that $y \in Y_{\chi}$ is equal to
$|Z/Z^x|$.  Hence, for $k = -n$ and all $y\in Y$, we have $\ch^{k}_y(\ck) = 0$.  As this equality is also
true for $k\neq-n$ by \eqref{e.cohzero}, it holds for all integers $k$.  Since $\ck$ is a direct sum of $IC$-sheaves,
this is only possible if $\ck = 0$.  Assertion (1) of the
proposition follows.  Similarly, if $y \not\in Y_{\chi}$, we have
$\dim \ch^{-n}_y (IC(Y, \cl_{\chi})) = 0$.  Since we have already
proved that  if $y \in Y_{\chi}$, then $\dim \ch^{-n}_y (IC(Y, \cl_{\chi})) =1$, 
assertion (2) follows.  Since $\dim \ch^{k}_y (IC(Y, \cl_{\chi})) = 0$ for $k \neq -n$
by \eqref{e.cohzero}, assertion (3) follows as well.
Finally, assertion (4) is a consequence of (3) and the fact that the restriction of $IC(Y,\cl_{\chi})$ to
$Y_{\chi}$ is $\cl_{\chi}[n]$.
\end{proof}

\subsection{The subvarieties $\toric_{\chi}$ and $\toric_{ad, \chi}$} \label{s.toric}
In this section, we consider the setup from Section~\ref{s.local}, in the case that $X=\toric$, $Y=\toric_{ad}$, and $Z$ is the center of $G=SL_n(\C)$.  We write $\toric_{ad, reg}= (\toric_{ad})_{reg}$ and $\toric_{ad, \chi}= (\toric_{ad})_{\chi}$ .  The goal of this section is 
to give an explicit description of 
$\toric_{\chi}$ and $\toric_{ad, \chi}$ for each $\chi\in \widehat{Z}$ (see Proposition \ref{prop.wchi} and Corollary \ref{cor.wchi}).

In this setting $Z = \{\omega^k I_n \} \cong \Z_n$ where
$\omega = \exp(2 \pi i/n)$ and $I_n$ is the $n \times n$ identity
matrix.  We use the notation $\omunder^k = \omega^k I_n$.  Let $\chi_k: Z \to \C^*$ be the character defined by
$\chi_r(z) = z^k$ for $z \in Z$.

\begin{Lem} \label{lem.chir}
\begin{enumerate}
\item Let $r = n/\gcd(n,c)$.  Then $\ker \chi_c = \{ 1, \omunder^r, \omunder^{2r}, \ldots \}$.  
\item Suppose $\chi \in \widehat{Z}$ has order $d$.  Then $\ker \chi = \ker \chi_{n/d} = 
\{ 1, \omunder^d, \omunder^{2d}, \ldots \}$ .
\end{enumerate}
\end{Lem}

\begin{proof}
(1) $\ker \chi_c$ is the cyclic group generated by the lowest power of $\omunder$ in $\ker \chi_c$.  Since
the smallest positive integer $k$ such that $n | ck$ is $k=r$, the lowest power of $\omunder$ in $\ker \chi_c$ is~$\omunder^r$.

(2) By definition we have $\chi^k(\omunder) = \chi(\omunder^k)$,
so $\chi^d = 1$ implies $\{1,\omunder^d, \omunder^{2d}, \ldots\}\subset \ker\chi$.
Conversely, suppose $\omunder^k\in \ker \chi$, but $k$ is not a multiple
of $d$.  Then $1 \leq m = \mbox{gcd}(k,d) < d$ 
and $\omunder^m \in \ker \chi$, so $\chi^m = 1$.  This contradicts
the assumption that the order of $\chi$ is~$d$.  
\end{proof}

Let $\ga_1, \ldots, \ga_{n-1}$ be the simple roots associated to the Lie algebra $\fg=\mathfrak{sl}_n(\C)$, and $\gl_1, \ldots, \gl_{n-1}$ denote the fundamental
dominant weights (we follow the conventions of \cite[Section 13.1]{HumphreysLieAlg}).  
We let $\mu_k$ denote the unique weight
in the same coset of $\gl_i$ (modulo the root lattice) such that $\mu_k = \sum_{i = 1}^{n-1} a_{ki} \ga_i$,
where the $a_{ki}\in \QQ$ satisfy
$0 \leq a_{ki} < 1$ for all $k$.  
We say $\ga_c$ occurs in $\mu_k$ if $a_{kc} \neq 0$.

Recall that $\gl_k$ is expressed as a linear combination of simple roots via the formula 
\begin{eqnarray} \label{e.glk}
\gl_k &= & \frac{1}{n} \Big( (n-k) \ga_1 + 2(n-k) \ga_2 + \cdots +
k(n-k) \ga_k  \\ 
\nonumber & & \quad\quad + k(n-k-1) \ga_{k+1} + k(n-k-2) \ga_{k+2} + \cdots + k \ga_{n-1} \Big).
\end{eqnarray}
We can use this formula to compute $\mu_k$, as demonstrated in the next example.

\begin{example} \label{ex.sl(12)}
Let $n=12$.
We can view the weights as elements of $\R^{12}$ whose
coordinates add up to $0$, and each simple root is
$\ga_i = \gre_i - \gre_{i+1}$ for $1\leq i \leq n-1$.  Using equation~\eqref{e.glk} we compute the cosets of the fundamental dominant weights modulo the root lattice.  For example,
\begin{eqnarray*}
\lambda_1&=& \frac{1}{12} (11\alpha_1+10\alpha_2+9\alpha_3+8\alpha_4+7\alpha_5+6\alpha_6+5\alpha_7+4\alpha_8+3\alpha_9+2\alpha_{10}+\alpha_{11})\\
\lambda_2&=& \frac{1}{12} (10\alpha_1+20\alpha_2+18\alpha_3+16\alpha_4+14\alpha_5+12\alpha_6+10\alpha_7+8\alpha_8+6\alpha_9+4\alpha_{10}+2\alpha_{11})\\
\lambda_3 &=& \frac{1}{12} (9\alpha_1+18\alpha_2+27\alpha_3 + 24\alpha_5+21\alpha_5+18\alpha_6+15\alpha_7+ +12\alpha_8+9\alpha_9+6\alpha_{10}+3\alpha_{11})\\
\lambda_{4}&=& \frac{1}{12}(8\alpha_1+16\alpha_2+24\alpha_3+32\alpha_4+28\alpha_5+24\alpha_6+20\alpha_7+16\alpha_8+12\alpha_9+8\alpha_{10}+4\alpha_{11})\\
\lambda_{5}&=& \frac{1}{12}(7\alpha_1+14\alpha_2+21\alpha_3+ 28\alpha_4+35\alpha_5+30\alpha_6+25\alpha_7+20\alpha_8+15\alpha_9+10\alpha_{10}+5\alpha_{11})\\
\lambda_{6}&=& \frac{1}{12}(6\alpha_1+12\alpha_2+18\alpha_3+24\alpha_4+30\alpha_5+36\alpha_6+30\alpha_7+24\alpha_8+18\alpha_9+12\alpha_{10}+6\alpha_{11});
\end{eqnarray*}
the fundamental weights $\gl_7, \ldots, \gl_{11}$ are given by similar formulas.
Each $\mu_i$ is obtained from $\gl_i$ by taking the fractional
portion of the coefficients $a_{ik}$:
\begin{eqnarray*}
\mu_1 &=& \lambda_1\\
\mu_2 &=& \frac{1}{6}(5\alpha_1 + 4\alpha_2+3\alpha_3+2\alpha_4+\alpha_5+ 5\alpha_7+4\alpha_8+3\alpha_9 + 2\alpha_{10}+\alpha_{11})\\
\mu_3 &=& \frac{1}{4}(3\alpha_1 + 2\alpha_2+\alpha_3+3\alpha_5+2\alpha_6+\alpha_7+3\alpha_9 + 2\alpha_{10}+\alpha_{11})\\
\mu_4 &=& \frac{1}{3}(2\alpha_1+\alpha_2 + 2\alpha_4+ \alpha_5 + 2\alpha_7+\alpha_8+2\alpha_{10}+\alpha_{11})\\
\mu_5 &=& \frac{1}{12}(7\alpha_1 + 2\alpha_2+9\alpha_3+4\alpha_4+11\alpha_5+6\alpha_6+ \alpha_7+8\alpha_8+3\alpha_9+ 10\alpha_{10}+5\alpha_{11})\\
\mu_6 &=& \frac{1}{2}(\alpha_1 + \alpha_3+\alpha_5+ \alpha_7+\alpha_9 +\alpha_{11}).
\end{eqnarray*}
We can compute $\mu_i$ for $7\leq i \leq 11$ using the same methods. In this case we find:
\begin{itemize}
\item $\ga_1$ and $\ga_{11}$ occur in all $\mu_i$
\item $\ga_2$ and $\ga_{10}$ occur in all $\mu_i$ except $\mu_6$
\item $\ga_3$ and $\ga_9$ occur in all $\mu_i$ except $\mu_4, \mu_8$
\item $\ga_4$ and $\ga_8$ occur in all $\mu_i$ except $\mu_3, \mu_6, \mu_9$
\item $\ga_5$ and $\ga_7$ occur in all $\mu_i$,
\item and $\ga_6$ occurs in all $\mu_i$ for odd $i$.
\end{itemize}
\end{example}

\vspace*{.05in}

Given a weight $\gl$ of $T$, we write $e^{\gl}$ for the corresponding
function on $T$.  If $\gl$ is in the root lattice,
$e^{\gl}$ can also be viewed as a function on $T_{ad}$.
Set $x_i = e^{-\ga_i}$ and $v_i= e^{-\mu_i}$ for all $1\leq i \leq n-1$. By definition, $\toric = \spec A$, where $A = \C[v_1, \ldots, v_{n-1}, x_1, \ldots, x_{n-1}]/I$ for some ideal $I$,
and $\toric_{ad} = \spec \C[x_1, \ldots, x_{n-1}] \cong
\C^{n-1}$.  The ideal $I$ contains all elements of the form 
\[
v_1^{a_1}v_2^{a_2}\cdots v_{n-1}^{a_{n-1}} - x_1^{b_1}x_2^{b_2}\cdots x_{n-1}^{b_{n-1}}
\]
for all $a_1,a_2, \ldots, a_{n-1}, b_1,b_2, \ldots, b_{n-1}\in \Z$ such that 
\[
a_1\mu_1+a_2\mu_2+\cdots +a_{n-1}\mu_{n-1} = b_1 \alpha_1 +b_2\alpha_2+\cdots+b_{n-1}\alpha_{n-1}.
\]
We view $v_1, \ldots, v_{n-1}$, $x_1, \ldots,
x_{n-1}$ as coordinates on $\C^{2n-2}$, and $\toric$
as the subscheme of $\C^{2n-2}$ defined by the ideal $I$. 

Given a point $x$ in $\toric$, we
write $v_i(x)$ and $x_i(x)$ for the values of the coordinate
functions $v_i$ and $x_i$ at $x$.  Similarly, if $y \in \toric_{ad}$,
$x_i(y)$ is the value of the coordinate $x_i$ at $y$. Recall that $\pi:\toric\to\toric_{ad}$ is the quotient by $Z$. We have $\pi(x) = y$ if and only if $x_i(x) = x_i(y)$ for all $i$.
If $z \in Z$, we have $\pi(z x) = \pi(x)$, so $x_i(z x) = x_i(x)$ for all $i$.

The center $Z$ acts on $x\in \C^{2n-2}$ by
\begin{eqnarray*}
&&z(v_1(x), v_2(x), \ldots, v_{n-1}(x), x_1(x), \ldots, x_{n-1}(x)) \\
&& \quad \quad\quad\quad\quad =   (z v_1(x), z^2 v_2(x), \ldots, z^{n-1} v_{n-1}(x), x_1(x), \ldots, x_{n-1}(x)).
\end{eqnarray*}
The action of $Z$ on the coordinate functions $v_i$ and $x_i$
is given by $z \cdot v_i = z^{-i} v_i$ and 
$z \cdot x_i = x_i$, respectively.   The $Z$-action on $\C^{2n-2}$ preserves the subscheme $\toric$ since the
ideal $I$ is $T$-invariant.

\begin{example} We continue Example~\ref{ex.sl(12)} for $n=12$.  Rewriting $\mu_i$ for $i=1,2,4$ as a function on $T$ and expanding in terms of $v_i$ and $x_i$, we find that 
\begin{eqnarray*}
	v_1^{12} &=& x_1^{11}x_2^{10}x_3^{9}x_4^8x_5^7x_6^6x_7^5x_8^4x_9^3x_{10}^2x_{11}\\
	v_2^{6} &=& x_1^5 x_2^4x_3^3x_4^2x_5 x_7^5 x_8^4x_9^3 x_{10}^2 x_{11}\\
	v_4^3 &=& x_1^2 x_2 x_4^2 x_5 x_7^2 x_8 x_{10}^2 x_{11}.
\end{eqnarray*}
Using similar methods, we obtain
\begin{eqnarray*}
	v_3^4&=& x_1^3 x_2^2 x_3 x_5^3 x_6^2 x_7 x_9^3 x_{10}^2 x_{11}\\
	v_5^{12}&=& x_1^7 x_2^2 x_3^9 x_4^4 x_5^{11} x_6^6 x_7 x_8^8 x_9^3 x_{10}^{10} x_{11}^5\\
	v_6^2 &=& x_1x_3x_5x_7x_9x_{11}.
\end{eqnarray*}
\end{example}

\vspace*{.05in}

The next lemma tells us when $\ga_c$ occurs in $\mu_k$; the reader can verify this result when $n=12$ using the data provided in Example~\ref{ex.sl(12)}.

\begin{Lem} \label{lem.occurs}
Let $c \in \{1, \ldots, n-1 \}$ and
let $r = n/\gcd(n,c)$.  Then $\ga_c$ occurs in
$\mu_k$ \textup{(}for $1 \leq k \leq n-1$\textup{)} if and only if $k$ is not a multiple of $r$.
\end{Lem}
\begin{proof}
We prove the equivalent statement that $\ga_c$ does not occur in $\mu_k$ if and only
if $r$ divides $k$.  
The root $\ga_c$ does not occur in $\mu_k$ if and only if
the coefficient of $\ga_c$ in $\mu_k$ is an integer.
Equation~\eqref{e.glk} shows that if $c \leq k$ then this coefficient is equal to
$c(n-k)/n$; if $c >k$ then this coefficient is
equal to $[k(n-k -(c-k))]/n = k(n-c)/n$.
In either case, the coefficient is an integer if and only if $n$
divides $ck$.  Writing $s = \gcd(n,c)$, we have
$n = rs$ and $c = ms$, where $r$ and $m$ are relatively prime.
Then $n = rs$ divides $ck = kms$ if and only if $r$ divides
$km$ if and only if $r$ divides $k$.
\end{proof}

Note that if $r = n/\gcd(n,c)$, the lemma implies that $\ga_c$ occurs in each $\mu_i$. Our first step toward computing $\toric_\chi$ is to compute the stabilizer groups $Z^x$ for $x\in \toric$.

\begin{Lem} \label{lem.stab}
For $x\in \toric$, $Z^x = \cap \ker \chi_k$ where the intersection is over all $k\in \{1,\ldots, n-1\}$ such that $v_k(x) \neq 0$.
\end{Lem}

\begin{proof}
Given $x, x' \in \toric$, we have $x = x'$ if and only if $x_k(x) = x_k(x')$  and $v_k(x) = v_k(x')$ for all $k$.  
Hence, for $z \in Z$, we have $z \in Z^x$ if and only if $x_k(zx) = x_k(x)$ and $v_k(z x) = v_k(x)$ for all $k$.
The first equality always holds, by the discussion above.  We have $v_k(zx) = \chi_k(z) v_k(x)$, and this
equals $v_k(x)$ if and only if $\chi_k(z) = 1$ whenever $v_k(x) \neq 0$, which is equivalent to the assertion
of the lemma.
\end{proof}

Motivated by the preceding lemma, we now examine the condition $v_k(x) = 0$.

\begin{Lem} \label{lem.vk=0}
Suppose $x \in \toric$.  We have $v_k(x) = 0$ if and only if $x_i(x) = 0$ for some $i$ such that $\ga_i$ occurs in
$\mu_k$.
\end{Lem}

\begin{proof}
We have $\mu_k = \sum_{i = 1}^{n-1} a_{ki} \ga_i$ for $a_{ki}\in \mathbb{Q}$ such that $0\leq a_{ki}<1$.
Choose a positive integer $d$ such that  $d a_{ki} \in \Z$ for all $i$.
Then $d \mu_k = \sum_{i = 1}^{n-1} d a_{ki} \ga_i$, so in the ring
$A$, we have 
\begin{equation} \label{e.vk}
v_k^d = \prod_{i=1}^{n-1} x_i^{d a_{ki}}.
\end{equation}
The root $\ga_i$ occurs in $\mu_k$ if and only if $a_{ki} > 0$.
So evaluating both sides of \eqref{e.vk} at $x$, we see that if 
$\ga_i$ occurs in $\mu_k$ and $x_i(x) = 0$, then
$v_k(x)^d = 0$, and hence $v_k(x) = 0$.  On the other hand,
if $x_i(x) \neq 0$ for all $i$ such that $\ga_i$ occurs in $\mu_k$,
then $(v_k(x))^d \neq 0$, so $v_k(x) \neq 0$.
\end{proof}

The next proposition describes the stabilizer groups $Z^x$ for $x \in \toric$.

\begin{Prop} \label{prop.stab1}
Let $x \in \toric$, and $\{ c_1, c_2, \ldots, c_{\ell} \}$ be the set of integers $k$ such that $x_{k}(x) = 0$.  
Let $r_i = n/\mbox{gcd}(n,c_i)$ and $r = \mbox{lcm}(r_1, r_2, \ldots, r_\ell)$.  \textup{(}Note that $r$ divides $n$.\textup{)}  Then~$Z^x = \ker \chi_r$.
\end{Prop}

\begin{proof}
By Lemma~\ref{lem.vk=0}, $v_k(x) = 0$ if and only if $x_i(x) = 0$ for some $i$ such that $\ga_i$ occurs in $\mu_k$. Equivalently, $v_k(x)=0$ if and only if $\alpha_{c_i}$ occurs in $\mu_k$ for some $i$.  Finally, by Lemma~\ref{lem.occurs} $\alpha_{c_i}$ occurs in $\mu_k$ if and only if $k$ is not a multiple of $r_i$. We therefore conclude that $v_k(x)\neq 0$ if and only if $k$ is a multiple of $r$. The desired statement now follows from Lemma~\ref{lem.stab}.
\end{proof}

Let $\chi:  Z \to \C^*$ be a character of $Z$. Recall that 
$\toric_{\chi}=\{ x \in \toric \mid Z^x \subset \ker \chi \}$.
If $\chi$ has order $n$ then $\ker \chi = \{ 1 \}$, so
$\toric_{\chi} = \toric_{reg}$.  The main result of this section is the next proposition, which
describes $\toric_{\chi}$.

\begin{Prop} \label{prop.wchi}
Suppose $\chi \in \widehat{Z}$ has order $d$ and let $x \in \toric$.  Then $x \in \toric_{\chi}$ if and only if
for each $k$ in $\{1, \ldots, n-1 \}$, 
$x_k(x) \neq 0$ if $d$ does not divide $k$.
\end{Prop}

\begin{proof}
Let $x \in \toric$, and
let $c_1, \ldots, c_{\ell}$ be the integers $k$ with $x_k (x) = 0$.  The proposition is equivalent to the assertion
that $x \in \toric_{\chi}$ if and only if  $d$ divides $c_i$ for all $i$.  

Proposition~\ref{prop.stab1} says that $Z^x = \ker \chi_r$, where $r = \mbox{lcm}(r_1, r_2, \ldots, r_\ell)$ for $r_i = n/\mbox{gcd}(n,c_i)$.  Applying this fact together with Lemma~\ref{lem.chir} we now have:
\begin{eqnarray*}
x\in \toric_\chi \Leftrightarrow Z^x\subseteq \ker\chi = \ker \chi_{n/d} \Leftrightarrow \textup{$r$ divides $\frac{n}{d}$} \Leftrightarrow \textup{$r_i$ divides $\frac{n}{d}$ for all $i$.}
\end{eqnarray*}
Finally, we have that $r_i$ divides $\frac{n}{d}$ for all $i$ if and only if $d$ divides $\gcd(n,c_i)$ for all $i$.  The last condition is equivalent to requiring that $d$ divide $c_i$ for all $i$, as desired.
\end{proof}

\begin{Cor} \label{cor.wchi}
Suppose $x \in \toric$.  Then $x \in \toric_{reg}$ if and only if $x_k(x) \neq 0$ for each $k\in\{1, \ldots, n-1 \}$. 
\end{Cor}

\begin{proof}
By definition, $x \in \toric_{reg}$ if and only if $Z^x = \{ 1 \}$.  Since $\{ 1 \} = \ker \chi$ when $\chi$ has
order $n$, Proposition \ref{prop.wchi} implies that $\toric_{reg}$ consists of the set of $x \in \toric$ such that 
$x_k(x) \neq 0$ for all $k\in\{1, \ldots, n-1 \}$.
\end{proof}

\begin{Cor} \label{cor.wadchi}
Suppose $\chi \in \widehat{Z}$ has order $d$.  Then $y \in \toric_{ad, \chi}$ if and only if 
for each $k$ in $\{1, \ldots, n-1 \}$, 
$x_k(y) \neq 0$ if $d$ does not divide $k$.
In particular,
$\toric_{ad, reg}$ consists of the set of $y \in \toric_{ad}$ such that
 $x_k(y) \neq 0$ for all $k\in\{1, \ldots, n-1 \}$.
\end{Cor}

\begin{proof}
By definition, $y$ is in $\toric_{ad, \chi}$ if $y = \pi(x)$ for $x \in \toric_{\chi}$.  In this case, $x_k(y) = x_k(x)$, so the
result follows from Proposition \ref{prop.wchi} and Corollary \ref{cor.wchi}.
\end{proof}

\subsection{The subvarieties $\widetilde{\cm}_{\chi}$ and $\widetilde{\cn}_{\chi}$} \label{s.nchi}
Recall that we have $\eta: \widetilde{\fu} \to \fu$, defined by
$$
\eta = \pi \times 1: \widetilde{\fu} = \toric \times_{\toric_{ad}} \fu \to \fu =  \toric_{ad} \times_{\toric_{ad}} \fu.
$$
The $Z$-action on $\widetilde{\fu}$ comes from the action on the first factor, and we have
$\widetilde{\fu}/Z = \fu$.  As varieties,
$\widetilde{\fu} = \toric \times \fg_{\geq 4}$ and $\fu = \toric_{ad} \times \fg_{\geq 4}$.  Both $\widetilde{\fu}$ and
$\fu$ have $B$-actions and the map $\map$ is $B$-equivariant.   From this we obtain
$$
\widetilde{\eta}: \widetilde{\cm} = G \times^B \widetilde{\fu} \to \widetilde{\cn} = G \times^B \fu.
$$
The $Z$-action on $\widetilde{\cm}$ is the restriction of the left $G$-action to the subgroup $Z$; note
that for $z \in Z$, $g \in G$, and $x \in \toric$, we have
$$
z [g,x] = [zg, x] = [gz, x] = [g,zx],
$$
where the second equality is because $z$ is central, and the
third holds since $Z \subset B$.  The equality $\widetilde{\fu}/Z = \fu$ implies that $\widetilde{\cm}/Z = \widetilde{\cn}$ (see
\cite{Gra:19}).

\begin{Prop} \label{prop.nchi} Let $\chi \in \widehat{Z}$.
\begin{enumerate}
\item We have $ \widetilde{\fu}_{\chi} = \toric_{\chi} \times_{\toric_{ad}} \fu $ and $ \fu_{\chi} = \toric_{ad,\chi} \times_{\toric_{ad}} \fu$.  As varieties,
$\widetilde{\fu}_{\chi} = \toric_{\chi} \times \fg_{\ge 4}$ and $\fu_{\chi} = \toric_{ad, \chi} \times \fg_{\ge 4}$.

\item We have $\widetilde{\cm}_{\chi} = G \times^B  \widetilde{\fu}_{\chi}$ and  $\widetilde{\cn}_{\chi} = G \times^B \fu_{\chi}$.
\end{enumerate}
\end{Prop}

\begin{proof}
(1) Suppose  $\widetilde{u} = (x, u) \in  \widetilde{\fu}$, where $x \in \toric$ and $u \in \fu$.   Then
$Z^{(x,u)} = Z^x$, so $\widetilde{u} \in \widetilde{\fu}_{\chi}$ if and only if $x \in \toric_{\chi}$.  This proves that $ \widetilde{\fu}_{\chi} = \toric_{\chi} \times_{\toric_{ad}} \fu $.  Since $\fu_{\chi}$ is the image of $ \widetilde{\fu}_{\chi}$ under the map 
$\eta$, it follows that $ \fu_{\chi} = \toric_{ad,\chi} \times_{\toric_{ad}} \fu$.  The assertions about the structure
of $\widetilde{\fu}_{\chi}$ and $\fu_{\chi}$ as varieties follow from writing $\fu = \toric_{ad} \times \fg_{\ge 4}$.

(2) If $[g, \widetilde{u}] \in \widetilde{\cm}$, then $Z^{[g, \widetilde{u}]} = Z^{\widetilde{u}}$, so  
$[g, \widetilde{u}] \in \widetilde{\cm}_{\chi}$ if and only if $u \in \widetilde{\fu}_{\chi}$.  Hence $\widetilde{\cm}_{\chi} = G \times^B  \widetilde{\fu}_{\chi}$.  The assertion $\widetilde{\cn}_{\chi} = G \times^B \fu_{\chi}$ follows because
$\widetilde{\cn}_{\chi}$ is the image of $\widetilde{\cm}_{\chi}$ under $\widetilde{\eta}$.
\end{proof}

\begin{example}\label{ex: n-chi} Let $n=6$ and $\chi=\chi_3\in \widehat{Z}$; note that $\chi$ is a character
of order $d=2$.  In this example, we consider the intersections $\tilde{\cn}_\chi \cap \mu^{-1}(\nilp)$ for various nilpotent elements~$\nilp\in \cn$.  Here $\mu^{-1}(\nilp)$ is the Springer fiber over $\nilp$; recall from \eqref{e.SpringerBorbit} that
$ \mu^{-1}(\nilp) = \{ [g, g^{-1} \nilp] \mid g^{-1}\cdot \nilp \in \fu\}$.  The previous proposition shows that $\tilde{\cn}_\chi$ is determined by $\fu_{\chi}$: we have
$ [g, g^{-1} \nilp] \in \cn_{\chi}$ if and only if $g^{-1} \nilp \in \fu_{\chi}$.  We have
$\fu_{\chi} = p^{-1}(\toric_{ad, \chi})$, where $p: \fu \to \toric_{ad}=\fu/[\fu,\fu]$ is the natural projection. 
Write
\begin{eqnarray}\label{eqn: projection formula}
p(g^{-1}\cdot \nilp) = c_1E_{\alpha_1} + c_{2}E_{\alpha_2}+ c_3E_{\alpha_3} + c_4 E_{\alpha_4}+ c_5 E_{\alpha_5}.
\end{eqnarray}
We deduce from Proposition~\ref{prop.wchi} that
$(g, g^{-1}\cdot\nilp)\in \tilde{\cn}_\chi$ if and only if $c_k\neq 0$ for all odd~$k$.  

It is well-known that the irreducible components of $\mu^{-1}(\nilp)$ for $\nilp\in \co_\lambda$ are indexed by standard tableaux of shape $\lambda$. Indeed, if we construct these components as Steinberg does in~\cite[\S 2]{Steinberg}, each standard tableau $S\in \mathrm{ST}(\lambda)$ corresponds to an open subset $C_S\subset \mu^{-1}(\nilp)$ of maximal dimension.  The decomposition of $\mu^{-1}(\nilp)$ into irreducible components is  $\mu^{-1}(\nilp) = \cup_{S\in \mathrm{ST}(\lambda)}\overline{C}_S$. In this
example, it is straightforward to show that $C_S \subset \tilde{\cn}_\chi$  whenever $S$ is one of the following standard tableaux.
\begin{eqnarray}\label{eqn: standard tab n=6}
\ytableausetup{centertableaux}
\begin{ytableau}
1 & 2 \\
3 & 4 \\
5 & 6 \\
\end{ytableau}
\quad\quad
\begin{ytableau}
1 & 2 & 3 & 4 \\
5 & 6 \\
\end{ytableau}
\quad\quad
\begin{ytableau}
1 & 2 & 5 & 6 \\
3 & 4 \\
\end{ytableau}
\quad\quad
\begin{ytableau}
1 & 2 & 3 & 4 &  5 & 6 \\
\end{ytableau}
\end{eqnarray}
This computation is discussed in greater detail for the first standard tableaux above in Example~\ref{ex.mainthm}.
\end{example}

\vspace*{.05in}

Corollary \ref{cor.wadchi} implies that $\fu_{reg}$ consists of the elements of $\fu$
such that the coefficient of each root vector corresponding
to a simple root is nonzero.  This set is
exactly the set of principal nilpotent elements in $\fu$.
Therefore, $\tilde{\cn}_{reg}$ is the set of $[g,\nilp]\in G\times^B \fu$ where
$\nilp$ is principal nilpotent in $\fu$.  The map $\mu: [g,\nilp] \to g\cdot \nilp$
identifies $\tilde{\cn}_{reg}$ with the principal nilpotent orbit
$\co^{pr}$ in $\cn$.  As discussed in Section~\ref{sec.local-system-nilp-orbit}, each character $\chi$ of $Z$ induces
a local system $\cl_{\chi}$ on $\co^{pr}$ and consequently on $\tilde{\cn}_{reg}$.  We are now in the setting of Section~\ref{s.local}, with $X=\tilde{\cm}$ and $Y=\tilde{\cn}$.  Applying Proposition~\ref{prop.pushforward} to $\tilde{\map}: \tilde{\cm} \to \tilde{\cn}$ yields the following result.

\begin{Prop} \label{prop.pushforwardcn}
Let $\tilde{\map}: \tilde{\cm} \to \tilde{\cn}$ denote the map induced by $\map: \tilde{\fu} \to \fu$.  Then
$$
\tilde{\map}_*\shq_{\tilde{\cm}}[\dim \cn] = \bigoplus_{\chi \in \widehat{Z}} IC(\tilde{\cn}, \cl_{\chi}).
$$
Moreover, if $i$ is the inclusion of
$\tilde{\cn} \setminus \tilde{\cn}_{\chi}$ into 
$\tilde{\cn}$, then $i^* IC(\tilde{\cn}, \cl_{\chi})=0$. 
\end{Prop}


\section{$IC$ sheaves on $\tilde{\cn}$ and $\tilde{\cn}^P$}\label{sec: parabolic ic sheaves}  In this section we continue our study of the IC-complexes $IC(\tilde{N}, \cl_\chi)$ for $\chi\in \widehat{Z}$.  The main result of the section is
Theorem~\ref{thm.cnppushforward}, which shows that the pushforward of this complex to the
variety $\tilde{\cn}^P$ 
studied in \cite{Borho1983} 
is equal to the IC-complex $IC(\tilde{\cn}^P, \cl_\chi)$.

We continue to assume $G = SL_n(\C)$.  Throughout this section, we fix
positive integers $d$ and $r$ such that $n = dr$.
Let $L_d$ be the Levi subgroup of $G$ containing $T$ whose simple
roots are the $\ga_k\in \Delta$ such that $k$ is not
divisible by $d$.   Then $L_d$ is a block diagonal matrix with $d \times d$ blocks;
that is, 
$$
L_d \cong S(GL_d(\C) \times GL_d(\C) \times \cdots \times GL_d(\C)),
$$
with $r$ factors on the right hand side.
We let $P_d$ denote the standard 
parabolic subgroup of $G$ with Levi factor $L_d$ and unipotent radical
$U_d$.  Since $B$ is upper triangular, $P_d$ is block upper triangular. 

Let  $I_d$ denote the
identity matrix in $GL_d(\C)$.
The center $Z(L_d)$ of 
$L_d$ consists of the block diagonal matrices
where the $i$-th block is $a_i I_d$ for $a_i\in \C$, subject to the condition
$(a_1 a_2 \cdots a_r)^d = 1$.  For simplicity,
we denote such a matrix by $z_{d}(a_1, \ldots, a_r)$.
The identity component
$Z(L_d)_0$ is the subset of elements of $Z(L_d)$
satisfying $a_1 a_2 \cdots a_r = 1$.  Therefore, the
component group  $C_d = Z(L_d)/Z(L_d)_0$ is isomorphic to 
$\Z_d$. 
The character group $\widehat{C_d}$ is generated
by the character 
$$
\phi: z_{d}(a_1, \ldots, a_r) \mapsto a_1a_2 \cdots a_r.
$$  
The inclusion map $Z \to Z(L_d)$ induces
a surjection $Z \to C_d$, with kernel
$\{ 1, \omunder^d, \omunder^{2d}, \ldots \}$.  
Pullback via this surjection yields an injective
map $\widehat{ C_d} \to \widehat{Z}$.

\begin{Lem} \label{lem.orderd}
\begin{enumerate}
\item The map $\widehat{ C_d} \to \widehat{Z}$
takes $\phi$ to $\chi_r$.  

\item If $\chi \in \widehat{Z}$ is a character of order $d$, then $\chi$ is 
is the image of an element of $\widehat{ C_d}$.
\end{enumerate}
\end{Lem}

 \begin{proof}
 (1) holds by direct calculation.  (2) follows since any element $\chi$ of $\widehat{Z}$ of order $d$ is a power
 of $\chi_r$, which is the image of an element of $\widehat{ C_d}$.
\end{proof}

\begin{Rem} \label{rem.character}
By abuse of notation, we will frequently use the same letter $\chi$ for an element of $\widehat{ C_d} $ 
and its image in  $\widehat{Z}$.  
\end{Rem}

For the remainder of this section, since $d$ is fixed, we simplify the notation by writing $P = LU_P$ for
$P_d = L_d U_d$.  We write $U_L = U \cap L$.   We let $\chi \in \widehat{Z}$ denote a character of order $d$.
We may repeat these assumptions for emphasis.

Let $\nilp_d$ be the matrix which has entries equal to $1$ above
the diagonal in each $d \times d$ block, and zeroes elsewhere.  In other words, $\nilp_d$ is the sum of simple root vectors corresponding to $\alpha_k\in \Delta$ such that $k$ is not divisible by $d$.  Thus, $\nilp_d$ is a principal nilpotent element of $\fl$. 
The next proposition describes the stabilizer groups $P^{\nilp_d}$ and $L^{\nilp_d}$ and the corresponding component
groups.

\begin{Lem} \label{lem.stabilizer}
\begin{enumerate}
\item  $L^{\nilp_d} = Z(L) U_L^{\nilp_d}$.

\item $P^{\nilp_d} = L^{\nilp_d} U_P^{\nilp_d}$.

\item The inclusions $Z(L)\subset L^{\nilp_d} \subset P^{\nilp_d}$ induce identifications of component groups
$$
Z(L)/Z(L)_0 = L^{\nilp_d} / (L^{\nilp_d})_0 = P^{\nilp_d} / (P^{\nilp_d})_0.
$$
\end{enumerate}
\end{Lem}

\begin{proof} 
Statement (1) follows from a straightforward computation of the stabilizer of a principal nilpotent element in $GL_d(\C)$, which we omit.
We prove (2). Suppose $p = u \ell \in P^{\nilp_d}$ with $u\in U_P$ and $\ell\in L$.  For any $x \in \fl$, the image of
$u \ell x$ under the projection
$\fp \to \fl$ equals $\ell x$.  Since $u \ell \nilp_d = \nilp_d \in \fl$, we see that $\ell \nilp_d = \nilp_d$, so
$\ell \in L^{\nilp_d}$.  This implies that $u \in U_P^{\nilp_d}$, proving (2).   Finally, the equalities in (3) follow 
from the fact that the groups $U_P^{\nilp_d}$ and $U_L^{\nilp_d}$ are unipotent, and therefore connected.  
\end{proof}

We have a Levi decomposition 
$\fp = \fl \oplus \fu_P$.  Let $\cn_L$ denote
the nilpotent cone in $\fl$.  Recall from Section~\ref{sec.genspringer} that
$$
\tilde{\cn}^P = G \times^P (\cn_L + \fu_P);
$$
we write $[g,x]_P$ for the element of $\tilde{\cn}^P$ corresponding to $(g,x)$ for 
$g \in G$, $x \in \cn_L + \fu_P$.  Borho and MacPherson defined a stratification
of  $\tilde{\cn}^P$ indexed by $L$-orbits
on the nilpotent cone $\cn_L$; the stratum corresponding
to the orbit $\co_L \subset \cn_L$ is $G \times^P (\co_L + \fu_P)$ (see \cite[Section 2.10]{Borho1983}; note
that $\co_L + \fu_P$ is $P$-stable). The
stratum corresponding to the principal nilpotent orbit 
$\co_L^{pr}$ in $\cn_L$ is denoted 
$\tilde{\cn}^P_{reg}$.  

\begin{Lem} \label{lem.fuintersect}
Suppose $\chi \in \widehat{Z}$ has order $d$, and let $P = P_d$,
$L = L_d$.  Then $\fu \cap (\co_L^{pr} + \fu_P) = \fu_{\chi}$.
\end{Lem}

\begin{proof}
Write $x \in \fu$ as a sum of root vectors:
$$
x = \sum_{i = 1}^{n-1} a_{\ga_i} X_{\ga_i} + \sum_{\gb \in \Phi^+ - \Delta} a_{\gb} X_{\gb},
$$
where the $\ga_i$ are the simple roots and $a_\gamma\in \C$ for all $\gamma\in \Phi^+$.  We have $x \in \co_L^{pr} + \fu_P$ if and only if for each
simple root $\ga_i$ of $L$,  $a_{\ga_i} \neq 0$.
The simple roots of $L$ are the $\ga_i$ where
$i$ is not divisible by $d$.  By Propositions \ref{prop.wchi} and 
\ref{prop.nchi}, $x$ is in $\fu_{\chi}$ if and only
if $a_{\ga_i}$ is nonzero when $i$ is not divisible by $d$.
Hence, $x$ is in $\fu_{\chi}$ if and only if
$x$ is in $\fu \cap (\co_L^{pr} + \fu_P)$.
\end{proof}

There is a semismall map $\rho_P: \tilde{\cn} \to \tilde{\cn}^P$ \cite[Lemma 2.10(e)]{Borho1983}, defined as follows.
Given $g \in G$ and $x \in \fu$ (resp.~$x \in \cn_L + \fu_P$),
temporarily write $[g,x]_B$ (resp.~ $[g,x]_P$) for the equivalence class of $(g,x)$ in
$\tilde{\cn} = G \times^B \fu$ 
(resp.~ in $\tilde{\cn}^P = G \times^P (\cn_L + \fu_P)$).  By definition, $\rho_P([g,x]_B) = [g,x]_P$.

\begin{Prop} \label{prop.partialisom}
Suppose $\chi \in \widehat{Z}$ has order $d$, and let $P = P_d$.
The map $\rho_P: \tilde{\cn} \to \tilde{\cn}^P$ satisfies
$\rho_P^{-1}(\tilde{\cn}^P_{reg}) = \tilde{\cn}_{\chi}$.
Moreover, $\rho_P$ takes $\tilde{\cn}_{\chi}$ isomorphically
onto~$\tilde{\cn}^P_{reg}$.
\end{Prop}

\begin{proof}
Since $\co_L^{pr}+\fu_P$ is $B$-invariant (as it is $P$-invariant), the intersection
$(\co_L^{pr}+\fu_P) \cap \fu$ is $B$-invariant.  The definition of $\rho_P$ implies
that
$$
\rho_P^{-1}(\tilde{\cn}^P_{reg}) = G \times^B ((\co_L^{pr}+\fu_P) \cap \fu). 
$$
By Lemma~\ref{lem.fuintersect}, $(\co_L^{pr}+\fu_P) \cap \fu= \fu_\chi$, so
$$
\rho_P^{-1}(\tilde{\cn}^P_{reg}) = G \times^B \fu_{\chi} = \tilde{\cn}_{\chi},
$$
where the last equality is by Proposition~\ref{prop.nchi}.  This proves
the first assertion of
the proposition.

We now show that $\rho_P$ takes $\tilde{\cn}_{\chi}$ isomorphically
onto~$\tilde{\cn}^P_{reg}$.
Borho and MacPherson \cite[Lemma 2.10(b)]{Borho1983} show that 
the fiber of $\rho_P$ over the stratum corresponding to $\co_L$ is isomorphic
to the Springer fiber $\cb_L^x$; here $\cb_L$ is the flag
variety for $L$, and $x$ is an element of $\co_L$.
For the stratum corresponding to the principal orbit 
$\co^{pr}_L$,
the Springer fiber $\cb_L^x$ is a single point.
Hence $\rho_P$ induces a bijection
\begin{equation} \label{e.bijection}
\tilde{\cn}_{\chi} = \rho_P^{-1}(\tilde{\cn}^P_{reg}) \to \tilde{\cn}^P_{reg}.
\end{equation}
Since $\tilde{\cn}^P_{reg}$ is smooth, 
Zariski's Main Theorem implies that~\eqref{e.bijection}
is an isomorphism of schemes, completing the proof.
\end{proof}

Observe that $\co^{pr}$ can be viewed as an open dense subset of any
of $\cn$, $\tilde{\cn}$, or $\tilde{\cn}^P$.
Indeed, the Springer resolution $\mu:\tilde{\cn} \to \cn$ is an isomorphism over
$\co^{pr}$; the map $\mu$ factors through $\rho_P: \tilde{\cn} \to \tilde{\cn}^P$, and $\rho_P$ takes the
open set $\tilde{\cn}_{\chi}$ (which contains $\co^{pr}$)
isomorphically onto its image.  A character $\chi$ of $Z$ corresponds to a local system
$\cl_{\chi}$ on $\co^{pr}$, and we obtain IC complexes
$IC(\tilde{\cn}, \cl_{\chi})$ and $IC(\tilde{\cn}^P, \cl_{\chi})$.

\begin{Thm} \label{thm.cnppushforward}
Suppose $\chi \in \widehat{Z}$ has order $d$, and let $P = P_d$.
Then $(\rho_P)_* IC(\tilde{\cn}, \cl_{\chi}) = IC(\tilde{\cn}^P, \cl_{\chi})$.
\end{Thm}

\begin{proof}
By the Decomposition Theorem (see Theorem \ref{decomp}), $(\rho_P)_* IC(\tilde{\cn}, \cl_{\chi})$ is a
direct sum of shifted $IC$ complexes on $\tilde{\cn}^P$.  We claim that the shifts are all trivial.  To prove the claim,
it suffices
to show that $(\rho_P)_* IC(\tilde{\cn}, \cl_{\chi})$ is perverse, since no nontrivial shift of 
an intersection cohomology complex remains perverse (see \cite[p.~37]{Borho1983}).  Observe
that the composition
$\rho_P \circ \widetilde{\eta}: \tilde{\cm} \to \tilde{\cn}^P$ is semismall,
since $\rho_P$ is semismall and $\widetilde{\eta}$ is finite.  Moreover, 
$\tilde{\cm}$ is rationally smooth, since it is locally a quotient of a smooth variety
by a finite group.
Therefore, by  \cite[Section 1.7]{Borho1983}, $(\rho_P \circ \widetilde{\eta})_* \shq_{\tilde{\cm}}$ is perverse.  
By Proposition \ref{prop.pushforwardcn}, 
$$
(\rho_P \circ \widetilde{\eta})_* \shq_{\tilde{\cm}} =(\rho_P)_* (\oplus_{\chi \in \widehat{Z}} IC(\tilde{\cn}, \cl_{\chi}) ).
$$
Since the right hand side satisfies the support condition to be a perverse sheaf,
each summand also satisfies the same support condition
and therefore each summand is also perverse.  The claim follows.

The proper base change theorem implies that the pushforward $(\rho_P)_*$ commutes with the restriction to the
open set $\tilde{\cn}_{\chi} \cong
\tilde{\cn}^{P}_{reg}$.  Since $\rho_P$ is an isomorphism over
this open set, and the local system $\cl_{\chi}$ extends
to this open set,
the perverse sheaf $(\rho_P)_* IC(\tilde{\cn},\cl_{\chi})$
restricts to $\cl_{\chi}[n]$ on $\tilde{\cn}^{P}_{reg}$ (where
$n = \dim \cn$).  Arguing as in the proof of Proposition \ref{prop.pushforward}, we see 
that the only intersection cohomology complex on $\tilde{\cn}^P$ whose
restriction to $\tilde{\cn}^{P}_{reg}$ is $\cl_{\chi}[n]$ is $IC(\tilde{\cn}_P,\cl_{\chi})$.
Therefore, $IC(\tilde{\cn}^P,\cl_{\chi})$
must occur as a summand in $(\rho_P)_* IC(\tilde{\cn},\cl_{\chi})$.
Hence we can write
$$
(\rho_P)_* IC(\tilde{\cn},\cl_{\chi}) = IC(\tilde{\cn}^P,\cl_{\chi}) \oplus \ck,
$$
where $\ck$ is a direct sum of intersection cohomology complexes
on subvarieties of $\tilde{\cn}^P$.  To complete the proof, we must show that
$\ck = 0$.

As a step towards this, we claim that 
if $IC(Z,\cf)$ occurs in $\ck$, then
 $Z$ must be contained in $\tilde{\cn}^P \setminus \tilde{\cn}^{P}_{reg}$.  Indeed, if not, then
$Z \cap \tilde{\cn}^{P}_{reg}$ would be open and dense in $Z$,
so it would contain a point $z$ of the smooth open set where
the local system $\cf$ is defined.  
Therefore $\ch^*_z ( (\rho_P)_* IC(\tilde{\cn},\cl_{\chi}))$
would be at least $2$-dimensional, since it would have
contributions from $IC(\tilde{\cn}^P,\cl_{\chi})$ and
$IC(Z,\cf)$.  This contradicts the fact that
the restriction of $(\rho_P)_* IC(\tilde{\cn},\cl_{\chi}) $
to $\tilde{\cn}^{P}_{reg}$ is $\cl_{\chi}[n]$.  Therefore
$Z \subset \tilde{\cn}^P \setminus \tilde{\cn}^{P}_{reg}$,
proving the claim.

Consider the cartesian diagram
$$
\begin{CD}
\tilde{\cn} \setminus \tilde{\cn}_{\chi} @>i>> \tilde{\cn} \\
@VV{\rho_P'}V @VV{\rho_P}V\\
\tilde{\cn}^P \setminus \tilde{\cn}^{P}_{reg} @>j>> \tilde{\cn}^P
\end{CD}
$$
where $\rho_P'$ is the restriction of $\rho_P$.  By the proper base change theorem,
$$
j^* (\rho_P)_* IC(\tilde{\cn},\cl_{\chi}) = (\rho'_P)_* i^* IC(\tilde{\cn},\cl_{\chi}) = 0,
$$
where the second equality holds since $i^* IC(\tilde{\cn},\cl_{\chi}) = 0$ by Proposition \ref{prop.pushforwardcn}.
Since $\ck$ is a direct summand in $ (\rho_P)_* IC(\tilde{\cn},\cl_{\chi})$, we deduce that $j^* \ck = 0$.

We know that $\ck$, if nonzero, is a direct sum of
terms of the form $IC(Z,\cf)$, where $Z$ is a subvariety
of $\tilde{\cn}^P \setminus \tilde{\cn}^{P}_{reg}$.  Observe that
for such a $Z$,
following the conventions of Section \ref{sec.decomp}, we have
$j^* IC(Z,\cf) = IC(Z,\cf)$.  Such a complex has a nonzero stalk
at any point $z$ in the open set of $Z$ where the local system 
$\cf$ is defined.  Since $j^* \ck = 0$, we see that no term of the form
$IC(Z,\cf)$ can appear in $\ck$, and therefore $\ck = 0$, as desired.
\end{proof}


\section{The generalized Springer correspondence and Lusztig Sheaves} \label{s.induction}

Recall that our goal is to prove that $\psi_* \shq_{\tilde{\cm}}$ is a direct sum of Lusztig sheaves, which were defined in Section~\ref{sec.genspringer} above.  In order to proceed we need to realize each $IC(\tilde{\cn}^P, \cl_\chi)$ as an induced complex.  After some preliminary discussion, we prove our main theorem, which is Theorem~\ref{thm.genspringerA} below.

We keep the notation of Section \ref{sec: parabolic ic sheaves}.  Thus,
$G = SL_n(\C)$, $\chi \in \widehat{Z}$ has order
$d$, and $\nilp_d\in \cn$, $P_d = L_d U_{P_d}$ are as in Section \ref{sec: parabolic ic sheaves}.  
As in that section, we write $P = LU_P$ for $P_d = L_d U_{P_d}$.

We recall some notation from Section \ref{sec.genspringer}.  The projection $q: \fp = \fl + \fu_P \to \fl$ is $P$-equivariant,
where the $P$-action on the target is the extension of the adjoint $L$-action to $P$ by requiring that $U_P$ acts
trivially.  We use the same notation $q$ for the induced maps $\cn_L + \fu_P \to \cn_L$ and
$\co_L^{pr} + \fu_P \to \co_L^{pr}$, which are also $P$-equivariant.  The map $q$ induces a map
$\pi_P: \tilde{\cn}^P \to G\times^P \cn_L$.  

Note that if $x \in \fl$, even though $q(x) \in \fl$ is the same element as $x$,
the notation indicates that the $P$-action is different.  Indeed, if $u \in U_P$, then $u \cdot x = x+y$ for some
$y \in \fu_P$, but $u \cdot q(x) = q(x)$.  Given $x \in \fl$, we will write $\overline{x} = q(x)$ for the same element,
but with trivial action of $U_P$. Recall that $\nilp$ denotes a principal nilpotent element.

\begin{Lem} \label{lem:dimension}
$P \cdot \nilp$ is open in $\cn_L + \fu_P$.
\end{Lem}

\begin{proof}
Observe that $P \cdot \nilp$ has the same dimension as $\cn_L + \fu_P$.  Indeed, $P^\nilp = G^\nilp$ has dimension $n-1$, so $\dim P \cdot \nilp = \dim P - (n-1)$; on the
other hand, $\dim \cn_L = \dim L - (n-1)$, so $\dim \cn_L + \dim \fu_P = \dim L + \dim U_P - (n-1)
= \dim P - (n-1)$.  This verifies the assertion about dimensions.  Since  $\cn_L + \fu_P$ is irreducible, the closure of $P \cdot \nilp$
must equal $\cn_L + \fu_P$. Since any orbit is open in its closure (see 
\cite[Section 2.1]{Jantzen}), $P \cdot \nilp$ is open in $\cn_L + \fu_P$.
\end{proof}

\begin{Rem}
An alternative proof of the lemma is 
as follows.  We know that $\co^{pr} = G \cdot \nilp$ is open and dense in $\cn$.  Since
$\tilde{\cn}^P \to \cn$ is an isomorphism over $\co^{pr}$, we can identify $G \cdot \nilp$ with its inverse image
in $\tilde{\cn}^P$, and that inverse image is therefore open and dense in $\tilde{\cn}^P$.  We write $\co^{pr}_L$ for either the orbit $L \cdot \nilp_d$ or $L \cdot \overline{\nilp_d}$. The fiber over $eP$ of the projection
$\tilde{\cn}^P \to G/P$ is identified with $\co^{pr}_L + \fu_P$.  The intersection of $G \cdot \nilp$ with this
fiber is open and dense in the fiber, and is identified with $P \cdot \nilp$, completing the proof.
\end{Rem}

Lemma~\ref{lem.orderd} implies that the character $\chi$ of $Z$ is the pullback of
a character of $C_d$, which we again denote
by $\chi$ (cf.~Remark \ref{rem.character}).
By Lemma~\ref{lem.stabilizer}, the component
group $L^{\nilp}/L^{\nilp}_0$ of $L^{\nilp}$ is equal to
$C_d$.  Therefore, $\chi$  induces
a local system $\cl^L_{\chi}$ on the orbit $\co^{pr}_L$.  
The Lusztig sheaf $\mathbb A_\chi$ is defined by
$$
{\mathbb A}_\chi =  {\mu_P}_*\pi_P^*\Ind_P^G IC(\cn_{L},\cl^L_{\chi})[d_P],
$$
where $d_P = \dim\fu_P$.
Recall the map $\psi: \widetilde{\cm} \to \cn$, which factors as
$$
\widetilde{\cm} \xrightarrow{\;\;\tilde{\map}\;\;} \widetilde{\cn} \xrightarrow{\;\;\rho_P\;\;}
\widetilde{\cn}^P \xrightarrow{\;\;\mu_P\;\;} \cn.
$$

Our main theorem is the following.

\begin{Thm} \label{thm.genspringerA}
Let $\psi: \tilde{\cm} \to \cn$ be the extended Springer resolution.  Then
$$
\psi_* \shq_{\tilde{\cm}}[\dim\cn] = \bigoplus_{\chi \in \widehat{Z}} \A_{\chi}.
$$
\end{Thm}

\begin{proof}
By Proposition \ref{prop.pushforwardcn} and Theorem \ref{thm.cnppushforward}, we have
$$
(\rho_P \circ \tilde{\eta})_*  \shq_{\tilde{\cm}}[\dim\cn] = \bigoplus_{\chi \in \widehat{Z}} IC(\tilde{\cn}, \cl_{\chi}).
$$
Comparing with the definition of the Lusztig sheaf, we see that it suffices to prove that
as $G$-equivariant perverse sheaves on $\tilde{\cn}^P$, we
have
\begin{equation} \label{e.induction}
\pi_P^* \Ind_P^G(IC(\cn_L, \cl_\chi^L))[d_P] =
IC({\tilde{\cn}^P}, \cl_{\chi}).
\end{equation}
From the map $q:\co_L^{pr} + \fu_P \to \co_L^{pr}$ we obtain a local system $q^* \cl^L_{\chi}$ on
$\co_L^{pr} + \fu_P$.  We have
\begin{eqnarray*}
 \pi_P^* \Ind_P^G IC(\cn_L, \cl_\chi^L) [d_P] & = & \Ind_P^G (q^* IC(\cn_L, \cl_\chi^L) )[d_P] \\
& = & \Ind_P^G IC(\cn_L + \fu_P, q^* \cl^L_{\chi})[d_P] \\
& = & IC(\tilde{\cn}^P, \Ind_P^G( q^* \cl^L_{\chi})).
\end{eqnarray*}
Here the first equality is because $\pi_P = \Ind_P^G(q)$ and $\Ind_P^G$ is a functor; the second
equality is because of the compatibility of IC complexes with smooth pullback (see \cite[Lemma 2.15]{Jantzen}); and the third equality
is because induction equivalence is compatible with
the construction of IC-complexes (see~\cite[\S 5.2]{BLEquiv}).    To complete the proof, it
suffices to show that there is some $G$-stable open set of $\tilde{\cn}^P$ on which the restrictions of the local systems
$\Ind_P^G(q^* \cl^L_{\chi})$ and $\cl_{\chi}$ are isomorphic. This is equivalent to
 showing that there is some $P$-stable open set of  $\co^{pr}_L + \fu_P$ on which the restrictions of $q^* \cl^L_{\chi}$ 
 and $\cl_{\chi}$ are isomorphic.  We will verify this for the open set $P \cdot \nilp$ of $\co^{pr}_L + \fu_P$.
 
The map $q$ induces (by restriction) a map of orbits
\begin{equation} \label{e.orbitmap}
q: P \cdot \nilp \to P \cdot \overline{\nilp};
\end{equation}
here the notation $\overline{\nilp}$ is as discussed at the beginning of the section.  
We need to check that 
\begin{equation} \label{e.pullbackchi}
q^* (\cl^L_{\chi}|_{P \cdot \overline{\nilp}}) \cong \cl_{\chi}|_{P \cdot \nilp}.
\end{equation}
The map \eqref{e.orbitmap} corresponds to the map of component groups
$$
P^{\nilp}/P^{\nilp}_0 = G^{\nilp}/G^{\nilp}_0 \cong Z \to P^{\overline{\nilp}}/P^{\overline{\nilp}}_0 = L^{\nilp}/L^{\nilp}_0
= Z(L)/Z(L)_0.
$$
By construction, under this map of component groups the character $\chi$ of $C_d=Z(L)/Z(L)_0$ pulls back
to the character of $Z$ which we have also denoted by $\chi$.  This implies \eqref{e.pullbackchi}, and the result follows.
\end{proof}

Theorem \ref{thm.genspringerA} implies that we can use the geometry of $\tilde{\cm}$ to study the generalized Springer correspondence.  For example, $Z$ must permute the irreducible components of the ``generalized Springer fibers'' $\psi^{-1}(\nu)$ for each $\nu\in \cn$. This is referred to as a {monodromy representation}.  From the discussion in \cite[\S 1.2]{Borho1983}, it follows that the multiplicity of $IC(\overline{\co}_\lambda, \cl_\chi)$ in $\psi_* \shq_{\tilde{\cm}}[\dim\cn]$ is exactly the multiplicity of the irreducible $Z$-representation with character $\chi$ in the monodromy representation on the irreducible components of $\psi^{-1}(\nu)$ for $\nu\in \co_\lambda$.  Our next example is a continuation of Example~\ref{ex: n-chi}, and computes this monodromy representation in a few cases.

\begin{example}\label{ex.mainthm} Let $n=6$ and $\chi\in Z$ be the character of order $d=2$. 
Consider the following standard tableau of shape $\lambda=[2^3]$.
\[
S = \ytableausetup{centertableaux}
\begin{ytableau}
1 & 2 \\
3 & 4 \\
5 & 6 \\
\end{ytableau}
\]
Example~\ref{ex: n-chi} shows that $C_S \subseteq \tilde{\cn}_\chi$ where $C_S$ is the open subset of the Springer fiber $\mu^{-1}(\nilp_2)$ of maximal dimension constructed as by Steinberg in~\cite[\S 2]{Steinberg}.  Let $g_0B\in \mu^{-1}(\nilp_2)$ denote a generic element.  Then equation~\eqref{eqn: projection formula} becomes
\[
p(g_0^{-1}\cdot \nilp_2) = E_{\alpha_1} + c_{2}E_{\alpha_2}+ E_{\alpha_3} + E_{\alpha_5},
\]
and $p(g_0^{-1}\cdot \nilp_2)$ corresponds to the coordinate values $(1,c_2,1,0,1)\in \toric_{ad}$ (we have $\toric_{ad} = \spec \C[x_1, \ldots, x_{n-1}] \cong \C^{n-1}$ as in Section~\ref{s.toric}). This confirms $C_S\subset \tilde{\cn}_\chi$ since $c_k\neq 0$ for all odd $k$.  

Recall from Section~\ref{s.toric} that we identify $\toric$ with a subscheme of 
$$\C^{2n-2}\simeq \spec \C[x_1, \ldots, x_{n-1};v_1,\ldots, v_{n-1}].$$  
Note that $x\in \toric$ satisfies $\pi(x) = p(g_0^{-1}\cdot \nilp_2)$ if and only if $x$ has coordinates
\[
(1,c_2,1,0,1; 0,0,v_3,0,0) \textup{ where } v_3^2 = 1
\]
since $v_3^2= x_1x_3x_5$ and $\alpha_4$ occurs in $\mu_k$ for all $k\neq 3$ (so $v_k(x)=0$ for all $k\neq 3$).  In other words, there are exactly two points in $\tilde{\cm}$ over each point of $C_S\subset \tilde{\cn}$.  It is straightforward to show that $\tilde{\map}^{-1}(\overline{C}_S)$ consists of precisely two irreducible components (which are also irreducible components of $\psi^{-1}(\nilp_2)$), each corresponding to the two possible values for $v_3$, namely $v_3\in \{\pm 1\}$.   These irreducible components are permuted by the $Z$-action (given by $v_3\mapsto \omega^3v_3=-v_3$); yielding a $Z$-representation with character $\iota+\chi$, where $\iota$ is the trivial character.  The $Z$-action on all other irreducible components of $\psi^{-1}(\nilp_2)$ is trivial.

Continuing in this way, we can compute $\tilde{\map}^{-1}(\overline{C}_S)$ for the other standard tableaux appearing in~\eqref{eqn: standard tab n=6}.  Our computations show that when $S$ is either of
\[
\begin{ytableau}
1 & 2 & 3 & 4 \\
5 & 6 \\
\end{ytableau}
\quad\quad
\begin{ytableau}
1 & 2 & 5 & 6 \\
3 & 4 \\
\end{ytableau}
\]
then  $\tilde{\map}^{-1}(\overline{C}_S)$ consists of precisely two irreducible components, permuted by the $Z$-action as above. We also show that $Z$ acts trivially on all other irreducible components of $\psi^{-1}(\nilp)$ for $\nilp \in \co_{[4,2]}$.  

When $S$ is the standard tableau with a single row, then the corresponding nilpotent element $\nilp$ of $\cn$ is regular so $\mu^{-1}(\nilp)$ is a single point.  The fiber in $\tilde{\cm}$ over $\mu^{-1}(\nilp)$ consists of exactly $6$ points and $Z$ acts by the regular representation on these components of $\psi^{-1}(\nilp)$.

These monodromy computations show that we expect the simple $G$-equivariant sheaves corresponding to the local system $\cl_\chi$ to appear with a total multiplicity of $4$, since $\chi$ only appears in the $Z$-action on $\tilde{\map}^{-1}(\overline{C}_S)$ for $S$ from~\eqref{eqn: standard tab n=6}.  In particular:
\begin{itemize}
\item $IC(\overline{\co}_{[2^3]}, \cl_{\chi})$ has  multiplicity $1$,
\item $IC(\overline{\co}_{[4,2]}, \cl_{\chi})$ has multiplicity $2$, and
\item  $IC(\overline{\co}_{[6]}, \cl_{\chi})$ has multiplicity $1$.
\end{itemize}
The table in Figure~\ref{fig1} gives the generalized Springer correspondence for $SL_6(\C)$. The reader may confirm the information of the table in Figure~\ref{fig1} by computing the generalized Springer correspondence for $SL_6(\C)$ as discussed in Section~\ref{sec.typeA}.  Column 6 of the table shows that the multiplicities computed above using the monodromy action match those of the simple perverse sheaves appearing in the generalized Springer correspondence since $\dim(V_{[1,1,1]}) = 1$, $\dim(V_{[2,1]}) = 2$, and $\dim(V_{[3]}) = 1$. 

\end{example}

\vspace*{.05in}

The reader may note that there is a combinatorial connection between the standard tableaux appearing in~\eqref{eqn: standard tab n=6} and the partitions of $3$ indexing irreducible representations of the relative Weyl group $S_3$ appearing in column 6 of Figure~\ref{fig1}.  This pattern generalizes.  In a forthcoming paper, the authors give an explicit description of the irreducible components of $\psi^{-1}(\nilp)$ as $\nilp\in \cn$ varies and compute the monodromy representation.  

The discussion in Section~\ref{sec.genspringer} tells us that, in the type $A$ case, only the simple perverse sheaves corresponding to $G_{ad}$-equivariant local systems (that is, the trivial local systems) appear in the Springer correspondence.   In other types, there are $G_{ad}$-equivariant local systems which do not appear in the Springer correspondence and we do not expect the results of Theorem~\ref{thm.genspringerA} to hold in full generality.  On the other hand, many of the arguments above do hold in the general setting, and it is likely $\psi_* \shq_{\tilde{\cm}}[\dim\cn]$ is a sum of Lusztig sheaves in this case also.  The authors plan to study the geometric constructions of this manuscript for arbitrary reductive algebraic groups in future work.

\begin{figure}
\resizebox{6in}{!}{\begin{tabular}{|c|c||c|c|c|c|c|}
\hline
& Local & $L_d=T$ & $L_d\cong S(GL_3\times GL_3)$ & $L_d\cong S(GL_3\times GL_3)$ &$ L_d\cong S(GL_2\times GL_2\times GL_2) $& $L_d=G$ (Cuspidal)\\
Orbit & System& $S_{n/d} = S_6$ & $S_{n/d}  = S_2$ & $S_{n/d}  = S_2$ & $S_{n/d}  = S_3$ & $S_{n/d}$ trivial \\
\hline\hline
$[1^6]$ & Triv & $\dim V=1$& & & & \\
\hline
$[2\ 1^4]$ & Triv & $\dim V=5$& & & & \\
\hline
$[2^2\ 1^2]$ & Triv & $\dim V=9$& & & & \\
\hline
$[2^3]$ & Triv & $\dim V=5$& & & & \\
 & Order 2 & & & &$\dim V=1;\ [1^3]$ & \\
\hline
$[3\ 1^3]$ & Triv & $\dim V=10$& & & & \\
\hline
$[3\ 2\ 1]$ & Triv & $\dim V=16$& & & & \\
\hline
$[3^2]$ & Triv & $\dim V=5$& & & & \\
 & Order 3 & &$\dim V=1;\ [1^2]$  & && \\
 & Order 3 & & & $\dim V=1;\ [1^2]$& & \\
\hline
$[4\ 1^2]$ & Triv & $\dim V=10$& & & & \\
\hline
$[4\ 2]$ & Triv & $\dim V=9$& & & & \\
 & Order 2 & & & &$\dim V=2;\ [2\ 1]$ & \\
\hline
$[5\ 1]$ & Triv & $\dim V=5$& & & & \\
\hline
$[6]$ & Triv & $\dim V = 1$& & & & \\
& Order 6 & & & & &$\dim V = 1$ \\
& Order 3 & &$\dim V = 1;\ [2]$  & & &\\
& Order 2 & & & & $\dim V = 1;\ [3]$ & \\
& Order 3 & & &$\dim V = 1;\ [2] $  & & \\
& Order 6 & & & & &$\dim V = 1$ \\
\hline
\end{tabular}}
\begin{caption}{ The generalized Springer correspondence for $SL_6(\C)$.  Local systems on each orbit are identified by the order of the corresponding central character.  Each simple perverse sheaf has multiplicity given by dimension of an irreducible $S_{n/d}$ representation (where $d$ denotes the order of the character).
 }
\label{fig1}
\end{caption}
\end{figure}


\ifx\undefined\bysame
\newcommand{\bysame}{\leavevmode\hbox to3em{\hrulefill}\,}
\fi

\end{document}